\documentclass[a4paper]{amsart}

\calclayout

\usepackage[utf8]{inputenc}
\usepackage[english, swedish]{babel}
\usepackage{amsmath}
\usepackage{graphicx}

\usepackage{amsfonts}
\usepackage{siunitx}
\usepackage{float}
\usepackage{rotating}
\usepackage[section]{placeins}
\usepackage{ stmaryrd }
\usepackage{ marvosym }
\usepackage{mathrsfs}
\usepackage[all,cmtip]{xy}

\usepackage{amssymb,microtype}

\usepackage{amsthm}

\usepackage{thmtools, thm-restate}

\usepackage{tikz}
\usetikzlibrary{babel}
\usepackage{sseq}

\usepackage{microtype}

\usepackage{tikz-cd}

\usepackage[english, status=draft]{fixme}

\usepackage[toc,title]{appendix}

\usepackage[colorlinks=false]{hyperref}

\usepackage{amscd}

\fxusetheme{color}

\usetikzlibrary{matrix,arrows,decorations}

\DeclareMathOperator{\Ab}{Ab}

\DeclareMathOperator{\Gal}{Gal}

\DeclareMathOperator{\Art}{Art}
\DeclareMathOperator{\coker}{coker}

\DeclareMathOperator{\Ext}{Ext}

\DeclareMathOperator{\Pic}{Pic}
\DeclareMathOperator{\Tors}{Tors}

\DeclareMathOperator{\Drv}{D}

\DeclareMathOperator{\Ind}{Ind}

\DeclareMathOperator{\Cl}{Cl}

\DeclareMathOperator{\divis}{div}

\DeclareMathOperator{\Div}{Div}

\DeclareMathOperator{\Cone}{C}

\DeclareMathOperator{\Sh}{Sh}

\newcommand{\C}{\mathcal C}

\usepackage{amsmath,calligra,mathrsfs}

\DeclareMathOperator{\DIV}{\mathscr{D}\text{\kern -1pt {\normalfont\emph{iv}}}\,}
\DeclareMathOperator{\HOM}{\mathscr{H}\text{\kern -2.5pt {\normalfont\emph{om}}}}

\newcommand{\K}{\mathcal K}

\newcommand{\id}{\operatorname{id}}

\newcommand{\Hom}{{\operatorname{Hom}}}

\newcommand{\RHom}{R\Hom}

\newcommand{\RGamma}{R\Gamma}

\newcommand{\RHOM}{R\kern -2pt\HOM}

\newcommand\blfootnote[1]{%
  \begingroup
  \renewcommand\thefootnote{}\footnote{#1}%
  \addtocounter{footnote}{-1}%
  \endgroup
}

\newcommand{\Spec}{\normalfont \text{Spec }}

\newcommand{\OO}{{\mathcal O}}

\newcommand{\ZZ}{{\mathbb Z}}
\newcommand{\QQ}{{\mathbb Q}}
\newcommand{\FF}{{\mathbb F}}
\newcommand{\GG}{{\mathbb G}}

\makeatletter

\newcommand{\et}{\'et\@ifstar{\'e}{e\xspace}}

\newtheorem{thm}{Theorem}
\numberwithin{thm}{section}

\newtheorem{prop}[thm]{Proposition}
\newtheorem{lem}[thm]{Lemma}
\newtheorem{cor}[thm]{Corollary}
\theoremstyle{definition}
\newtheorem{df}[thm]{Definition}
\newtheorem{ex}[thm]{Example}
\theoremstyle{remark}
\newtheorem{rk}[thm]{Remark}

\newcommand{\mmu}{\pmb{\mu}}

\newcommand{\pr}{{\normalfont \text{pr}}}

\newcommand{\Xet}{X_{\normalfont \textsf{\'et}}}

\newcommand{\Yet}{Y_{{\normalfont\textsf{\'et}}}}

\newcommand{\Xcet}{\tilde{X}_{\normalfont \textsf{\'et}}}
\newcommand{\Ycet}{\tilde{Y}_{\normalfont \textsf{\'et}}}
\newcommand{\Yicet}{\tilde{(Y_i)}_{\normalfont \textsf{\'et}}}

\DeclareMathOperator{\Br}{Br}
\DeclareMathOperator{\inv}{inv}

\newcommand{\fl}{\normalfont \textsf{fl}}

\newcommand{\im}{{\normalfont\text{im}}}

\makeatletter
\newcommand{\pushright}[1]{\ifmeasuring@#1\else\omit\hfill$\displaystyle#1$\fi\ignorespaces}
\newcommand{\pushleft}[1]{\ifmeasuring@#1\else\omit$\displaystyle#1$\hfill\fi\ignorespaces}
\makeatother

\title{{The \'Etale cohomology ring of the ring of integers of a number field}}
\author{Eric Ahlqvist}
\author{Magnus Carlson}

\begin{document}

\selectlanguage{english}
\begin{abstract}
We compute the cohomology ring $H^*(X,\ZZ/n\ZZ)$ for $X$ the spectrum of the ring of integers of a number field $K$. As an application, we give a non-vanishing formula for an invariant defined by Minhyong Kim.
\end{abstract}

\thanks{The first author was supported by the Swedish Research Council 2015-05554. The second author was sponsored by the Knut and Alice Wallenberg Foundation 2017.0400}

\maketitle

\tableofcontents

\section{Introduction}
\blfootnote{The minor datasets generated during and/or analysed during the current study are available from the corresponding author on reasonable request. The code we used can be found in the following repositiory: \\
https://github.com/ericahlqvist/cup-products.}
\blfootnote{The authors have no conflicts of interest to declare that are relevant to the content of this article.}
Let $K$ be a number field and let $X = \Spec \OO_K$. In this article we compute the \'etale cohomology ring $H^*(X,\ZZ/n\ZZ)$\footnote{If $X$ is not totally imaginary, we use a ``modified \'etale cohomology'' that takes the infinite primes into account.} for arbitrary $n \geq 1$. As an application of this computation, we give a non-vanishing criterion for an abelian ``arithmetic'' Chern--Simons invariant developed by Minhyong Kim. The ring $H^*(X,\ZZ/n\ZZ)$ holds non-trivial arithmetic information: for $n=2$ it has been used by the second author and Tomer Schlank \cite{CarlsonSchlankUnramified} in the study of embedding problems, while Maire \cite{MaireUnramified} used it to study the unramified Fontaine--Mazur conjecture at $p=2$ and the $2$-cohomological dimension of Galois groups with restricted ramification. It is our hope that our computation of $H^*(X,\ZZ/n\ZZ)$ will find interesting applications in the not too distant future. 

We now move on to stating our results, and we assume for simplicity of exposition that $K$ is a totally imaginary number field. However, the theorems in the main text are stated for arbitrary number fields, under the added proviso that one works with \'etale cohomology on the Artin--Verdier site (see Definition \ref{def:Xcet}). Recall that Artin--Verdier duality allows us to identify $H^i(X,\ZZ/n\ZZ)$ with $\Ext^{3-i}_X(\ZZ/n\ZZ,\mathbb{G}_{m,X})^\sim$, where $^\sim$ denotes the Pontryagin dual. One can further identify $\Ext^1_X(\ZZ/n\ZZ,\mathbb{G}_{m,X})$ with the group $Z_1/B_1$, (see Corollary \ref{cor:valuesofext}) where
  \[
    \begin{split}
      Z_1 & = \{(a,\mathfrak{a}) \in K^\times \oplus \Div K : \divis(a)+n\mathfrak{a} = 0\}\,, \\
      B_1 & = \{(b^{-n},\divis(b)) \in K^\times \oplus \Div K : b \in K^\times \}\,,
    \end{split}
  \]
$\Div K$ is the group of fractional ideals, and $\divis(a)$ is the fractional ideal associated to $a$.
Hence we obtain the following list (see Corollary \ref{cor:valuesofext}):
  \[
    H^i(X,\ZZ/n\ZZ) = 
      \begin{cases} 
        \ZZ/n\ZZ & \mbox{ if } i= 0 \\
        (\Pic(X)/n)^\sim & \mbox{ if } i= 1 \\
        (Z_1/B_1)^\sim & \mbox{ if } i=2 \\
        \mmu_n(X)^\sim & \mbox{ if } i=3 \\
        0 &  \mbox{ if } i>3.
      \end{cases}
  \]

An element $x \in H^1(X,\ZZ/n\ZZ)$ is given by a $\ZZ/n\ZZ$-torsor $Y\to X$, which corresponds to a cyclic unramified extension $L/K$ of degree $d|n$ together with a choice of generator $\sigma \in \Gal(L/K)$.
The structure of the cup product map $H^1(X,\ZZ/n\ZZ) \times H^1(X,\ZZ/n\ZZ) \rightarrow H^2(X,\ZZ/n\ZZ)$ is given by the following Proposition.

\begin{prop}
Let $X = \Spec \OO_K$ be the ring of integers of a totally imaginary number field $K$, and identify $H^2(X,\ZZ/n\ZZ)$ with $\Ext^1_{X}(\ZZ/n\ZZ,\mathbb{G}_{m,X})^\sim$, where $^\sim$ denotes the Pontryagin dual. Let $x \in H^1(X,\ZZ/n\ZZ)$ be an element corresponding to a cyclic unramified extension $L/K$ of degree $d|n$ and a choice of generator $\sigma \in \Gal(L/K)$. For an element $y \in H^1(X,\ZZ/n\ZZ) \cong (\Cl K/n \Cl K)^\sim$ represented by an unramified cyclic extension $M/K$, we have that 
  \[
    \langle x \cup y,(a,\mathfrak{b})\rangle =  \langle y,N_{L|K}(I)^{n/d}+\frac{n^2}{2d} \mathfrak{b}\rangle
  \] 
where $(a,\mathfrak{b}) \in \Ext^1_{X}(\ZZ/n\ZZ,\mathbb{G}_{m,X})$ and $I \in \Div L$ is any fractional ideal such that $\mathfrak{b}^{n/d} \mathcal{O}_L =I-\sigma(I) + \divis(t)$ for some $t \in L^\times$ such that $N_{L|K}(t) =a^{-1}$. In particular, $\langle x \cup y,(a,\mathfrak{b})\rangle = 0$ if and only if $\frac{n^2}{2d} \mathfrak{b}+N_{L|K}(I)^{n/d}$ is in the image of $N_{M|K}$.
\end{prop}

It should be noted that the formula for this cup product is closely related to a pairing of McCallum--Sharifi, as discussed in Remark \ref{rk:sharifi}. Our computation of the map $H^2(X,\ZZ/n\ZZ) \times H^1(X,\ZZ/n\ZZ) \rightarrow H^3(X,\ZZ/n\ZZ)$ is captured by:

\begin{prop}
Let $X = \Spec \OO_K$ be the ring of integers of a totally imaginary number field $K$ and identify $H^2(X,\ZZ/n\ZZ)$ and $H^3(X,\ZZ/n\ZZ)$ with $\Ext^1_{X}(\ZZ/n\ZZ,\mathbb{G}_{m,X})^\sim$ and $\mu_n(K)^\sim$ respectively, where $^\sim$ denotes the Pontryagin dual. Let $x \in H^1(X,\ZZ/n\ZZ)$ be a element corresponding to a cyclic extension $L/K$ of degree $d|n$, unramified at all places, together with a choice of generator $\sigma \in \Gal(L/K)$. Let $\xi \in \mu_n(K)$ and choose $b \in L^\times$ such that $\xi^{n/d} = \sigma(b)/b$, and let $a \in K^\times$ and $\mathfrak{a} \in \Div K$ be such that $\mathfrak{a}\OO_L = \divis(b)$ and $a = b^{-n}$ in $L^\times$. If $y \in H^2(X,\ZZ/n\ZZ)$, then 
  \[
    \langle x \cup y,\xi\rangle =  \langle y,(a,\mathfrak{a})\rangle\,.
  \]
\end{prop}

Let $\inv\colon H^3(X,\mmu_n) \rightarrow \ZZ/n\ZZ$ be the invariant isomorphism used in Artin--Verdier duality, which is defined by viewing $H^3(X,\mmu_n)$ as a subgroup of $H^3(X,\mathbb{G}_{m,X})$, which is isomorphic to $\QQ/\ZZ$ via class field theory, and then identifying $\ZZ/n\ZZ$ with the $n$-torsion of $\QQ/\ZZ$. When $K$ contains all roots of unity, we may choose an isomorphism of \emph{small} \'etale sheaves $\ZZ/n\ZZ\to \mmu_n$ and think of elements in $H^1(X,\ZZ/n\ZZ)$ as equivalence classes of pairs $(v, \mathfrak{a})\in K^\times\oplus \Div(K)$ such that $\divis(v)+n\mathfrak{a}=0$ and such that $K(v^{1/n})$ is an unramified extension of $K$. Our formula for Kim's invariant is as follows.

\begin{prop}
Let $K$ be a totally imaginary number field containing a primitive $n$th root of unity and let $X= \Spec \OO_K$ be its ring of integers. Let $C_n=\ZZ/n\ZZ$, let $c_1 \in H^1(C_n,\ZZ/n\ZZ) \cong \Hom_{\Ab}(\ZZ/n\ZZ,\ZZ/n\ZZ)$ correspond to the identity and $c_2 = \beta_n(c_1) \in H^2(C_n,\ZZ/n\ZZ)$ be its image under the Bockstein homomorphism. Suppose that we have a continuous homomorphism $f\colon \pi_1(X,x) \rightarrow C_n$, corresponding to the unramified Kummer extension $L/K$, where $L=K(v^{1/n})$ for some $v \in K^\times$ such that there is an $\mathfrak{a} \in \Div K$ satisfying $n \mathfrak{a} = - \divis(v)$. Then Kim's invariant, $\inv(f_X^*(c)) \in \ZZ/n\ZZ$ vanishes if and only if $\mathfrak{a}$ is in the image of the norm map $N_{L|K} \colon \Cl L \rightarrow \Cl K$, i.e. if and only if $\Art_{L|K}(\mathfrak{a})=0$, where $\Art_{L|K}$ means the Artin symbol. 
\end{prop}

Lastly, the methods used to compute $H^*(X,\ZZ/n\ZZ)$ are similar to the methods used to compute $H^*(X,\ZZ/2\ZZ)$ in \cite{CarlsonSchlankUnramified}. However, there are some new insights needed in order to generalize the computation to arbitrary $n$.

\subsection{Organization}
In Section \ref{sec:background} we recall some material on the \'etale cohomology of a number field and the Artin--Verdier site. What we cover in this section is classical and can be found in \cite{BienenfeldEtale} and \cite{MazurNotes}. In Section \ref{sec:cup} we then move on to determine the structure of $H^*(X,\ZZ/n\ZZ)$. In Section \ref{sec:kim} we first recall the invariant defined by Minhyong Kim in \cite{KimArithmetic}, whereafter we state the non-vanishing criteria.

\section{Background on the \'etale cohomology of a number field} \label{sec:background}
Let $K$ be a number field and $X = \Spec \OO_K$ its ring of integers. In the beautiful paper \cite{MazurNotes}, Mazur investigates the \'etale cohomology of totally imaginary number fields. From the point-of-view of \'etale cohomology, $X$ behaves as a $3$-manifold and satisfies an arithmetic version of Poincaré duality, namely Artin--Verdier duality. This duality states that for any constructible sheaf $F$, the \'etale cohomology group $H^i(X,F)$ is Pontryagin dual to $\Ext^{3-i}_X(F,\mathbb{G}_{m,X})$, where $\mathbb{G}_{m,X}$ is the sheaf of units. For fields $K$ that are not totally imaginary, Artin--Verdier duality holds only modulo the $2$-primary part. To remedy this, one must instead consider constructible sheaves on a modified \'etale site of $X$ which takes the infinite primes into account. The purpose of this section is to recall some results from \cite{BienenfeldEtale} where the duality results we will need are proven. In an appendix to \cite{HaberlandGalois}, Zink removes the $2$-primary restriction as well, but works with a modified cohomology which we will not use. For a very readable account of Zink's results, we recommend \cite{ConradEtale}.

\subsection{The Artin--Verdier site of a number field} \label{subsec:artinsite}
In the following subsection we recall some of the main results that we need from \cite{BienenfeldEtale}. We emphasize that the results of this subsection (\ref{subsec:artinsite}) are not new, but due to the heavy amount of (possibly non-standard) notation, we include it to avoid confusion. 

If $X$ is the ring of integers of a number field, we let 
  \[
    X_\infty = \{x_1,\ldots, x_n\}
  \] 
be the infinite primes of $X$. An infinite prime is either a real embedding $K \rightarrow \mathbb{R}$ or a pair of conjugate complex embeddings of $K$ into $\mathbb{C}$. If $Y$ is a scheme that is \'etale over $X$, a real archimedean prime of $Y$ is a point $y\colon \Spec \mathbb{C} \rightarrow Y$ which factors through $\Spec \mathbb{R}$. If $y$ does not factor through $\Spec \mathbb{R}$, then by the conjugation action on $\mathbb{C}$ we obtain a point $\bar{y} \neq y$.
A complex prime of $Y$ is a pair of points $y_1,y_2 \colon \Spec \mathbb{C} \rightarrow Y$ such that $y_1 \neq y_2$ and $y_1 = \bar{y}_2$. Finally, we define $Y_\infty$ to be the set whose elements are the real and complex primes of $Y$.

\begin{df} \label{def:Xcet}
The Artin--Verdier site of $X$, denoted $\Xcet$, is the site with objects pairs $(Y,M)$, where $g \colon Y \rightarrow X$ is a scheme that is separated and \'etale over $X$, $M \subset Y_\infty$, and $g(M) \subset X_\infty$ is unramified, i.e, if $p \in M$ is a complex prime, then its image is complex as well. A morphism 
  \[
    f\colon(Y_1,M_1) \to (Y_2,M_2)
  \] 
in $\Xcet$ is a morphism of $X$-schemes such that $f(M_1) \subset M_2$. A family of morphisms $\{f_i\colon(Y_i,M_i) \rightarrow (W,N)\}_{i \in I}$ is a covering if $\cup_i f_i(Y_i) = W$ and $\cup_i f_i(M_i) = N$.
\end{df}

Note that any morphism $f\colon(Y_1,M_1) \rightarrow (Y_2,M_2)$ in $\Xcet$ has the property that $f\colon Y_1 \rightarrow Y_2$ is \'etale and that $f\colon M_1 \rightarrow M_2$ is unramified, i.e., if $p \in M_1$ is a complex prime, then $f(p)$ is complex as well. The fact that Definition \ref{def:Xcet} gives a site is found in \cite[Prop. 1.2]{BienenfeldEtale}. We define $\Sh(\Xcet)$ to be the category of abelian sheaves on $\Xcet$. It will be convenient to have a more concrete description of $\Sh(\Xcet)$. As above, we let $X_\infty = \{x_1,\ldots, x_n\}$ be the infinite primes of $X$. Fix a separable closure $\bar{K}$ of $K$. For each infinite prime $x_i$, we fix an extension $\widetilde{x_i}$ of $x_i$ to $\bar{K}$. We then let $I_{\widetilde{x_i}}$ be the decomposition group of $\widetilde{x_i}$ (note that $I_{\widetilde{x_i}} \cong \ZZ/2\ZZ$ if $x_i$ is real, and that $I_{\widetilde{x_i}} $ is trivial if $x_i$ is complex). Since $I_{\widetilde{x_i}} \subset \Gal(\bar{K}/K)$, if we let $j \colon \Spec K \rightarrow X$ be the map induced from the inclusion, we see that for any \'etale sheaf $F$ on $X$, the pull-back $j^*F$, viewed as a Galois module, has a natural action of $I_{\widetilde{x_i}}$, and that we thus can take the fixed points with respect to this action and form $(j^*F)^{I_{\widetilde{x_i}}}$. One can also view $x_i$ as giving an absolute value on $K$; let $K_{x_i}$ be the completion with respect to $x_i$. If then $i \colon \Spec K_{x_i} \rightarrow X$ is the natural map, then $(j^*F)^{I_{\widetilde{x_i}}}$ is isomorphic to the global sections of $i^*F$, and we will sometimes write $F(K_{x_i})$ to denote $(j^*F)^{I_{\widetilde{x_i}}}$.  

In preparation of the following definition, consider the category of finite sets over $X_\infty$, i.e., the category whose objects are given by pairs $(A,f)$ where $A$ is a finite set and $f \colon A \rightarrow X_\infty$ is a function, and where the morphisms between objects are given by commutative triangles. This category becomes a Grothendieck site if we define a covering to be given by surjective morphisms. We then let $\Sh(X_\infty)$ be the category of sheaves on this site. Note that to give a sheaf $F_\infty \in \Sh(X_\infty)$ is the same as giving a collection of abelian groups, $F_x$, one for each $x \in X_\infty$.

\begin{df} \label{def:triples}
The category $\mathcal{S}_X$ is the category whose objects are given by triples 
  \[
    S = (F_\infty,F,\{\sigma_{x}\}_{x \in X_\infty})
  \] 
where $F_\infty = \{F_\infty(x) \}_{x \in X_\infty} \in \Sh(X_\infty)$, i.e. $F_\infty$ is a product of abelian groups, one for each infinite prime, $F$ is an abelian sheaf on $\Xet$ and for $x \in X_\infty$,  $\sigma_{x} \colon F_\infty(x) \rightarrow (j^*F)^{I_{\tilde{x}}}$ is a morphism of abelian groups. A morphism 
  \[
    f \colon S_1 = (F_\infty,F,\{\sigma_x\}_{x \in X_\infty}) \rightarrow S_2 = (F'_\infty,F',\{\sigma'_x\}_{x \in X_\infty})
  \] 
is a pair of maps $f_1 \colon F_\infty \rightarrow F'_\infty, f_2\colon F \rightarrow F'$ commuting with $\sigma_x,\sigma'_x,  x \in X_\infty$, upon pulling back $f_2$ by $j \colon \Spec K \rightarrow X$.
\end{df}

The following proposition shows precisely that the above definition gives us a concrete description of $\Sh(\Xcet)$.

\begin{prop}[{\cite[Prop. 1.2]{BienenfeldEtale}}] \label{prop:BienenEquiv}
The category $\mathcal{S}_X$ defined above is equivalent to the category of abelian sheaves on $\Xcet$.
\end{prop}

Proposition \ref{prop:BienenEquiv} is the archimedean analogue of \'etale recollement, i.e., how one reconstructs $\Sh(Y)$ from $\Sh(Z)$ and $\Sh(U)$ with gluing data, where $Z$ is a closed subscheme of $Y$ with complement $U$ (see \cite[Thm. II.3.10]{MilneEtale}). 

We will from now on often identify the category $\Sh(\Xcet)$ with $\mathcal{S}_X$. The following proposition shows how $\Sh(\Xet)$, $\Sh(\Xcet)$ and $\Sh(X_\infty)$ relate to each other.

\begin{prop}[{\cite[Prop. 1.1]{BienenfeldEtale}}] \label{prop:BienenMorph}
There are geometric morphisms 
  \[
    \phi_\ast \colon \Sh(\Xet) \rightleftarrows \Sh(\Xcet)\colon \phi^\ast\,,
  \]  
  \[
    \kappa_*  \colon \Sh(X_\infty) \rightleftarrows \Sh(\Xcet)  \colon \kappa^*\,.
  \]
Let $F \in \Sh(\Xet)$, $S = (G_\infty,G,\{\sigma_x\}_{x \in X_\infty}) \in \Sh(\Xcet)$, and $K \in \Sh(X_\infty)$. Then
  \[
    \begin{split}
      \phi_\ast F & = (\{(j^*F)^{I_{\widetilde{x_i}}}\},F,\id),\\
      \phi^\ast (S) & = G\,,\\
      \kappa_* K & = (K,0,0)\,, \\
      \kappa^* S & = G_\infty\,.
    \end{split}
  \]
Further, $\phi^\ast$ has a left adjoint, denoted by $\phi_!$, while $\kappa_\ast$ has a right adjoint $\kappa^!$. These satisfy the following formulas:
  \[
    \begin{split}
      \phi_! (F) & = (0,F,0)\,, \\
      \kappa^!(S) & = \{\ker \sigma_x\}_{x \in X_\infty}\,.
    \end{split}
  \]
\end{prop}

\begin{rk}
The forgetful functor $\Xcet \rightarrow \Xet$ is a morphism of sites, and thus gives rise to a geometric morphism 
  \[
    \tilde{\phi}_*\colon \Sh(\Xet) \rightleftarrows \Sh(\Xcet)  \colon \tilde{\phi}^*\,.
  \] 
Under the identification of $\Sh(\Xcet)$ with a category of triples as in Proposition \ref{prop:BienenEquiv},  $\tilde{\phi}_*$ and $\tilde{\phi}^*$ are identified with the functors $\phi_*$ and $\phi^*$ respectively.
\end{rk}

Note that if $\underline{A} \in \Sh(\Xet)$ is the constant sheaf on $\Xet$ with value $A$, then $\phi_\ast(\underline{A})$ is isomorphic to the constant sheaf on $\Sh(\Xcet)$ with value $A$. 

If $L/K$ is an extension of number fields, let $Y = \Spec \OO_L$ and $X = \Spec \OO_K$. For each infinite prime $y \in Y_\infty$ lying over $x \in X_\infty$, choose the decomposition group $I_{\tilde{y}}$ such that $I_{\tilde{y}} \subset I_{\tilde{x}}$. We have a natural map $\pi\colon Y \rightarrow X$ and we will now define a push-forward functor $\pi_\ast\colon \Sh(\Ycet) \rightarrow \Sh(\Xcet)$ and a pull-back functor 
  \[
    \pi^\ast\colon \Sh(\Xcet) \rightarrow \Sh(\Ycet)\,.
  \] 
This will be done by identifying $\Sh(\Ycet)$ and $\Sh(\Xcet)$ with categories of triples as in Definition \ref{def:triples}. Note that given a map $\pi$ as above, we have geometric morphisms
  \[
    \pi_\ast \colon \Sh(Y_\infty) \rightleftarrows \Sh(X_\infty)  \colon \pi^\ast
  \] 
and 
  \[
    \pi_\ast  \colon \Sh(\Yet) \rightleftarrows \Sh(\Xet)\colon \pi^\ast\,,
  \] 
and we want to combine these to get a geometric morphism 
  \[
    \pi_\ast \colon \Sh(\Ycet) \rightleftarrows \Sh(\Xcet) \colon \pi^\ast\,.
  \]

Given an extension $L/K$ as above we have the following commutative diagram 
  \[
    \begin{tikzcd}   
      \Spec L \arrow[d,"\pi"] \arrow[r,"i"] & Y \arrow[d,"\pi"]\\  
      \Spec K \arrow[r,"j"] & X\,. 
    \end{tikzcd}
  \] 
If $(F_\infty,F,\{\sigma_y \}_{y \in Y_\infty}) \in \Sh(\Ycet)$, we see that to construct $\pi_\ast\colon \Sh(\Ycet) \rightarrow \Sh(\Xcet)$, we must, for each $x \in X_\infty$,  give a natural map 
  \[
    \sigma'_x\colon \pi_\ast(F_\infty)(x) = \bigoplus_{y/x} F_\infty(y) \rightarrow (j^*\pi_*F)^{I_{\tilde{x}}}
  \] 
and this is what is done in the discussion that follows. Since $\pi$ is finite \'etale, by looking at the stalk at a separable closure, we see that the natural map $j^*\pi_* F\to\pi_* i^* F$ is an isomorphism. For $x \in X_\infty$, 
  \[
    (\pi_*i^* F)^{I_{\tilde{x}}} \cong  (i^* F)(\Spec \bar{K}^{I_{\tilde{x}}} \otimes_K L)\,,
  \]
and since $\Spec \bar{K}^{I_{\tilde{x}}}\otimes_K L =  \amalg_{y /x} \Spec \bar{K}^{I_{\tilde{y}}}$, where $y$ ranges over the primes lying over $x$, we see that 
  \[
    (\pi_*i^*F)^{I_{\tilde{x}}} \cong \bigoplus _{y /x} (i^*F)^{I_{\tilde{y}}}\,.
  \] 
If $x \in X_\infty$, we let 
  $
    \theta_*  \colon \bigoplus_{y/x} (i^*F)^{I_{\tilde{y}}} \rightarrow (j^*\pi_*F)^{I_{\tilde{x}}}
  $ 
be the isomorphism that is the composite of the isomorphism 
  \[
    \bigoplus_{y/x} (i^*F)^{I_{\tilde{y}}} \rightarrow (\pi_*i^*F)^{I_{\tilde{x}}}
  \]
followed by the natural isomorphism $(\pi_*i^*F)^{I_{\tilde{x}}} \rightarrow (j^*\pi_*F)^{I_{\tilde{x}}}$. We now construct $\sigma'_x\colon \pi_\ast(F_\infty)(x) = \bigoplus_{y/x} F_\infty(y) \rightarrow (j^*\pi_*F)^{I_{\tilde{x}}}$ as the composite of 
  \[
    \bigoplus_{y/x} F_\infty(y) \xrightarrow{\oplus_{y/x} \sigma_y} \bigoplus_{y/x} (i^*F)^{I_{\tilde{y}}}
  \] 
with $\theta_*$. This allows us to construct the claimed push-forward, which we record in the following definition.

\begin{df} \label{def:push-forward}
Let $L/K$ be an extension of number fields, $X = \Spec \OO_K$, $Y = \Spec \OO_L$, and let $\pi \colon Y \rightarrow X$ be the natural projection. Further, let $j \colon \Spec K \rightarrow X$ and $i \colon \Spec L \rightarrow Y$ be the maps that are induced by inclusion. Denote by $X_\infty$ and $Y_\infty$ the infinite places of $X$ and $Y$ respectively. Then the push-forward functor
  \[
    \pi_\ast \colon \Sh(\Ycet) \rightarrow \Sh(\Xcet)
  \] 
takes $(F_\infty,F,\{\sigma_y\}_{y \in Y_\infty})$ to $(\pi_\ast F_\infty , \pi_\ast F, \{\sigma'_x\}_{x \in X_\infty})$ where $\pi_\ast F$ is the push-forward in $\Sh(\Xet)$, $\pi_\ast F_\infty(x) = \bigoplus_{y /x} F_\infty(y)$, and
  \[
    \sigma'_x \colon \pi_\ast F_\infty(x) = \bigoplus_{y/x} F_\infty(y) \rightarrow (j^* \pi_\ast F)^{I_{\tilde{x}}}
  \] 
is the map $\bigoplus_{y/x} F_\infty(y) \xrightarrow{\oplus_{y/x} \sigma_y} \bigoplus_{y/x} (i^*F)^{I_{\tilde{y}}}$ followed by the map $\theta_*$ just defined.
\end{df}

Having defined the push-forward functor $\pi_\ast\colon \Sh(\Ycet) \rightarrow \Sh(\Xcet)$ we now define the pull-back functor. Just as above, we want to combine the two pull-back functors $\pi^*\colon \Sh(\Xet) \rightarrow \Sh(\Yet)$ and $\pi^*\colon \Sh(X_\infty) \rightarrow \Sh(Y_\infty)$ to a functor 
  \[
    \pi^*\colon \Sh(\Xcet) \rightarrow \Sh(\Ycet)\,,
  \] 
To do this, we follow the strategy of Bienenfeld \cite{BienenfeldEtale}. The map $\theta_*$ above gives a natural isomorphism $\theta_*\colon  \pi_* \tau_Y \cong \tau_X \pi_*$, where $\tau_X\colon \Sh(\Xet)\to \Sh(X_\infty)$ is the functor which takes $F$ to 
  $
    \{(j^*F)^{I_{\tilde{x}}}\colon  x\in X_\infty\}
  $
and $\tau_Y\colon \Sh(\Yet)\to \Sh(Y_\infty)$ is the functor which takes $G$ to 
  $
    \{(i^*G)^{I_{\tilde{y}}}\colon y\in Y_\infty\}\,.
  $ 
If $F \in \Sh(\Xet)$, since $\pi_*$ is right adjoint to $\pi^*$, we have a natural unit morphism 
  $
    \eta_F\colon  F \rightarrow \pi_*\pi^*F\,.
  $
We now let 
  \[
    \theta^*\colon \pi^* \tau_X F \rightarrow \tau_Y \pi^* F
  \] 
be the map adjoint to the composite 
  \[
    \tau_X F \xrightarrow{\eta_F} \tau_X \pi_* \pi^* F \xrightarrow{\theta_*^{-1}} \pi_* \tau_Y \pi^* F\,.
  \]
It is clear that this gives a natural transformation $\theta^* \colon \pi^* \tau_X \Rightarrow \tau_Y \pi^*$. Given 
  \[
    F = (F_\infty,F, \{\sigma_x\}_{x \in X_\infty}) \in \Sh(\Xcet)\,,
  \] 
we define $\sigma'$ as the composite
  \[
    \sigma' \colon \pi^*(F_\infty) \xrightarrow{\pi^* (\Pi_x \sigma_x)} \pi^*(\tau_X F) \xrightarrow{\theta^*} \tau_Y( \pi^* F)\,.
  \]  
Evaluating $\sigma'$ at a point $y \in Y_\infty$, we get a map 
  \[
    \sigma'_y \colon \pi^*(F_\infty)(y) \rightarrow (i^* \pi^*F)^{I_{\bar{y}}}\,.
  \]
This allows us to define the pull-back, which we record in the following definition.

\begin{df} \label{def:pull-back}
Let $L/K$ be an extension of number fields, $X = \Spec \OO_K,Y= \Spec \OO_L$, and let $\pi\colon Y \rightarrow X$ be the natural projection and $j \colon \Spec K \rightarrow X, i \colon \Spec L \rightarrow Y$ be the maps induced by inclusion. Then the pull-back functor 
  \[
    \pi^\ast  \colon \Sh(\Xcet) \rightarrow \Sh(\Ycet)
  \] 
takes
$(F_\infty,F,\{\sigma_x\}_{x \in X_\infty})$ to 
  \[
    (\pi^\ast (F_\infty),\pi^\ast(F) ,\{\sigma'_y\}_{y \in Y_\infty})
  \] 
where $\sigma'_y \colon (\pi^\ast F_\infty)(y) \rightarrow (i^* \pi^* F)^{I_{\tilde{y}}}$ is defined as above.
\end{df}

\begin{prop}\label{prop:adjoint}
Let $L/K$ be a finite extension of number fields and let $\pi \colon Y = \Spec \OO_L \rightarrow \Spec \OO_K = X$ be the natural projection. Then the functor 
  \[
    \pi_*\colon \Sh(\Ycet) \rightarrow \Sh(\Xcet)
  \] 
is right adjoint to 
  \[
    \pi^*\colon \Sh(\Xcet) \rightarrow \Sh(\Ycet)\,.
  \] 
If further $L/K$ is unramified (at all places, including the infinite ones), then $\pi_*$ is left adjoint to $\pi^*$ as well.
\end{prop}

\begin{proof}
The first part follows from \cite[Lemma 1.9 and Proposition 1.11]{BienenfeldEtale}. To prove the second part, one can either use \cite[Proposition 1.11]{BienenfeldEtale} or calculate it directly, which we leave to the reader.
\end{proof}

\begin{rk} \label{rk:notconnected}
Suppose that we have an object $\pi \colon (Y,Y_\infty) \rightarrow (X,X_\infty)$ of $\Xcet$, and suppose further that $(Y,Y_\infty) \cong \amalg_{i=1}^n (Y_i,(Y_i)_\infty)$ where $Y_i = \Spec \OO_{L_i}$ for $L_i/K$ a finite extension of $K$. Let us set 
  \[
    \Sh(\Ycet) := \Pi_{i=1}^n \Sh(\Yicet)\,. 
  \]
One easily sees that this is the category of sheaves on a natural Artin--Verdier site associated to $(Y,Y_\infty)$. It is clear that if we let $\pi_i \colon Y_i \rightarrow X$ be the restriction of $\pi$ to $Y_i$, that we can, using the adjunctions 
  \[
    (\pi_i)^*\colon  \Sh(\Xcet) \rightleftarrows \Sh(\Yicet) \colon (\pi_i)_*
  \]  
define an adjunction 
  \[
    \pi^* \colon \Sh(\Xcet) \rightarrow \Sh(\Ycet) \colon \pi_*\,.
  \] 
The functor $\pi_*$ will then be the pushforward and $\pi^*$ the pullback. It also clear, for formal reasons, that if for each $i$, $(\pi_i)_*$ is also left adjoint to $(\pi_i)^*$, then the functor $\pi_*$ is left adjoint to $\pi^*$. We will sometimes need to use $\pi_*$ and $\pi^*$ when $Y$ is not connected in Section \ref{sec:cup}.
\end{rk}

Since $\Sh(\Xcet)$ is the category of sheaves on a site, it is a Grothendieck topos, so we have cohomology functors $H^i(\tilde{X},-) \colon \Sh(\Xcet) \rightarrow \Ab$, defined as the derived functors of the global sections functor. With notation as in Proposition \ref{prop:BienenMorph}, if $F \in \Sh(\Xcet)$, we want to relate $H^i(\tilde{X},F)$ with $H^i(X,\phi^* F)$. For any sheaf $F$ on $\Xet$ we have a short exact sequence 
  \[
    0 \rightarrow \phi_! F \rightarrow \phi_* F \rightarrow \kappa_*\kappa^* \phi_* F \rightarrow 0
  \] 
in $\Sh(\Xcet)$. If we take $F$ to be equal to $\ZZ$, and let $S \in \Sh(\Xcet)$, we get by applying $\Ext^i_{\tilde{X}}(-,S)$ to this short exact sequence a long exact sequence
  \[
    \cdots \rightarrow R^i\kappa^!(S) \rightarrow H^i(\tilde{X},S) \rightarrow H^i(X,\phi^*S) \rightarrow R^{i+1}\kappa^!(S) \rightarrow \cdots
  \]
which is the local cohomology sequence for $X_\infty$ (see \cite[Proposition 1.4]{BienenfeldEtale}).

\begin{lem} \label{lem:localcoh}
If $F \in \Sh(\Xet)$ is such that for each $x \in X_\infty$, $j^*F$ is a cohomologically trivial $I_{\tilde{x}}$-module, then $H^i(X,F) \cong H^i(\tilde{X},\phi_* F)$ for all $i\geq 0$.
\end{lem}

\begin{proof}
This follows from the local cohomology sequence for $X_\infty$ and \cite[Lemma 3.7]{BienenfeldEtale}. Indeed, in the latter lemma, it is shown that if $S = (G_\infty,G,\{\sigma_x\}_{x \in X_\infty}) \in \Sh(\Xcet)$, then 
  \[
    R^i\kappa^!(S) =  
    \begin{cases} 
      \bigoplus_{x \in X_\infty} \ker \sigma_x & \mbox{ if } i= 0 \\
      \bigoplus_{x \in X_\infty} \coker \sigma_x  & \mbox{ if } i= 1 \\
      \bigoplus_{x \in X_\infty} H^{i-1}(I_{\tilde{x}},j^*_x G) & \mbox{ if } i \geq 2\,. 
    \end{cases}
  \]
But if $F \in \Sh(\Xet)$, then since $\phi_*F$ is of the form $(F_\infty, F, {\sigma_x})$, where every $\sigma_x$ is the identity, it follows that $H^i(X_\infty,\kappa^! \phi_* F) = 0$ for $i=0,1$. The assumption that $j^*F$ is cohomologically trivial shows that $H^i(X_\infty,\kappa^! \phi_* F)=0$ for $i\geq 1$ and by the local cohomology sequence we are done.
\end{proof}

The following proposition shows that the cohomology ring $H^*(\tilde{X},\ZZ/n\ZZ)$ will always be isomorphic to the cohomology ring $H^*(X,\ZZ/n\ZZ)$ unless $n$ is even and $X$ has real places.

\begin{prop}\label{prop:cohagrees}
Let $X = \Spec \OO_K$ be the ring of integers of a number field. Then the cohomology rings $H^*(\tilde{X},\ZZ/n\ZZ)$ and $H^*(X,\ZZ/n\ZZ)$ are isomorphic if either $K$ is totally imaginary or if $n$ is odd.
\end{prop}

\begin{proof}
We apply Lemma \ref{lem:localcoh}. For each complex place $x$ of $K$, $I_{\tilde{x}}$ is trivial so that $j^*F$ is cohomologically trivial. If we consider a real place $x$, then $I_{\tilde{x}} \cong \ZZ/2\ZZ$, so we must show that $j^*(\ZZ/n\ZZ) \cong \ZZ/n\ZZ$ (with trivial Galois action) is cohomologically trivial as a $\ZZ/2\ZZ$-module if $n$ is odd, but this is obvious.
\end{proof}



\subsection{Artin--Verdier duality for general number fields} \label{subsec:artin}
We now move on to stating the duality result which will be needed for our later computation of the cup product. The notation in this subsection is the same as in Section \ref{subsec:artinsite}. We will denote by $^\sim$ the functor 
  \[
    \RHom_{\Ab}(-,\mathbb{Q}/\ZZ) \colon \Drv(\Ab)^{op} \rightarrow \Drv(\Ab)\,.
  \] 
Let $F \in \Drv(\Xcet)$ and denote by $\mathbb{G}_{m,X}$ the sheaf of units on $X$. Then the morphism
  \[
    A\colon \RGamma(\tilde{X},F) \rightarrow \RHom_{\tilde{X}}(F,\phi_\ast \mathbb{G}_{m,X})[3]^\sim\,,
  \] 
is defined to be the adjoint to the composition map
  \[
    \RHom_{\tilde{X}}(\ZZ, F) \times \RHom_{\tilde{X}}(F, \phi_\ast \mathbb{G}_{m,X}) \rightarrow \RHom_{\tilde{X}}(\ZZ , \phi_\ast \mathbb{G}_{m,X})
  \]
followed by the trace map $\RHom_{\tilde{X}}(\ZZ,\phi_\ast \mathbb{G}_{m,X}) \rightarrow \mathbb{Q}/\ZZ[-3]$ (see \cite[Prop. 2.7]{BienenfeldEtale}). The following is the version of Artin--Verdier duality that we need.

\begin{thm}[{\cite[Thm\ 5.1]{BienenfeldEtale}}] \label{thm:AVduality}
Let $F$ be a constructible sheaf on $\Xet.$ Assume that for each $x_i \in X_\infty$, $I_{\widetilde{x_i}}$ acts trivially on $j^* F$.  Then the map
  \[
    A \colon \RGamma(\tilde{X},\phi_\ast F) \rightarrow \RHom_{\tilde{X}}(\phi_\ast F,\phi_\ast \mathbb{G}_{m,X})[3]^\sim
  \] 
is an isomorphism in $\Drv(\Ab)$.
\end{thm}

Note that the hypothesis of the theorem is satisfied if $F$ is a locally constant constructible sheaf on $X$ that is split by a morphism $p \colon Y = \Spec \OO_L \rightarrow \Spec \OO_K =X$, such that $L/K$ is unramified at all places, including the infinite ones. The proof of the following lemma is just as in \cite[Lemma 4.1]{CarlsonSchlankUnramified}.

\begin{lem} \label{lemma:morphismdual}
Let $X = \Spec \mathcal{O}_K$ be the ring of integers of a number field and let $f\colon F \rightarrow G$ be a morphism between bounded complexes of constructible sheaves on $\Xet$ such that for each $x_i \in X_\infty$, $I_{\widetilde{x_i}}$ acts trivially on each term in $j^*F$ and $j^* G$. Then, the map
  \[
    \RGamma(\tilde{X},\phi_\ast F) \xrightarrow{\RGamma(\tilde{X},\phi_*f)} \RGamma(\tilde{X},\phi_ \ast G)
  \]
corresponds under Artin--Verdier duality to the map 
  \[
    \RHom_{\tilde{X}}(\phi_\ast F,\phi_\ast \mathbb{G}_{m,X})[3]^\sim \xrightarrow{\RHom_{\tilde{X}}(\phi_*f,\phi_\ast\mathbb{G}_{m,X})[3]^\sim} \RHom_{\tilde{X}}(\phi_\ast G, \phi_\ast \mathbb{G}_{m,X})[3]^\sim\,.
  \]
\end{lem}

We end this section by computing $H^i(\tilde{X},\ZZ/n\ZZ)$. By Theorem \ref{thm:AVduality}, and the fact that $\phi_\ast \ZZ/n\ZZ = \ZZ/n\ZZ$, this is the same as computing $\Ext^{3-i}_{\tilde{X}}(\ZZ/n\ZZ,\phi_\ast \mathbb{G}_{m,X}).$ For $i=0,1,3$ and $i >3$ this can be found in \cite[Prop. 2.13]{BienenfeldEtale}, but just as in \cite{MazurNotes}, $\Ext^1_{\tilde{X}}(\ZZ/n\ZZ,\phi_\ast \mathbb{G}_{m,X})$ is not explicitly determined. In the paper \cite{CarlsonSchlankUnramified}, the second author and Tomer Schlank gave a concrete interpretation when $X = \Spec \OO_K$ is the ring of integers of a totally imaginary number field, and we will now use the same method. The following presentation will be brief; for more details we refer to the aforementioned paper by the second author and Schlank.
Consider on $X$ the \'etale sheaf 
  \[
    \DIV(X) = \bigoplus_p \ZZ_{/p}\,,
  \]  
where we let $p$ range over all closed points of $X$ and $\ZZ_{/p}$ means that we consider the skyscraper sheaf at that point. We now define the complex $\mathcal{C}$, which is a resolution of $\mathbb{G}_{m,X}$, as
  \[ 
    j_\ast \mathbb{G}_{m,K} \xrightarrow{\divis} \DIV X\,,
  \] 
where $j_*$ is the inclusion of the generic point, $j_\ast \mathbb{G}_{m,K}$ is in degree $0$, and the map $\divis$ is as in \cite[II 3.9]{MilneEtale}. By push-forward with $\phi_\ast$, we get a complex $\phi_\ast \mathcal{C}$.
It is easy to see that the complex $\phi_\ast \mathcal{C}$ is a resolution of $\phi_\ast \mathbb{G}_{m,X}$. We now define $\mathcal{E}_n$ as the complex 
  \[
    \ZZ \xrightarrow{n} \ZZ
  \] 
of constant sheaves on $\tilde{X}$, where the non-zero terms are in degree $-1$ and $0$. This is of course just a resolution of $\ZZ/n\ZZ$, concentrated in degree $0$. Let $\HOM$ denote the internal hom in the derived category of abelian sheaves and consider the complex $\HOM(\mathcal{E}_n,\phi_\ast \mathcal{C})$, whose components are 
  \[ 
    \phi_\ast j_* \mathbb{G}_{m,K} \xrightarrow{\begin{pmatrix} -n \\ \divis\end{pmatrix}} \phi_\ast j_* \mathbb{G}_{m,K} \oplus \phi_\ast \DIV X \xrightarrow{\begin{pmatrix} \divis & n\end{pmatrix}}  \phi_\ast \DIV X\,.
  \]
Since $\mathcal{E}_n$ is a complex of free sheaves, we have 
  \[
    \HOM(\mathcal{E}_n,\phi_\ast \mathcal{C}) \cong \RHOM(\ZZ/n\ZZ,\phi_\ast \mathbb{G}_{m,X})
  \] 
in $\Drv(\Xcet)$. The plan for computing $\Ext^i_{\tilde{X}}(\ZZ/n\ZZ,\phi_\ast \mathbb{G}_{m,X})$ is to use the hypercohomology spectral sequence, applied to $\HOM(\mathcal{E}_n,\phi_\ast \mathcal{C})$. To do this we need the cohomology of $\phi_\ast j_* \mathbb{G}_{m,K}$ and $\phi_\ast \DIV X$ as input.

\begin{prop}[{\cite[Prop. 2.5, Prop. 2.6]{BienenfeldEtale}}] \label{prop:Bienencoh}
Let  $\phi_\ast j_* \mathbb{G}_{m,K}$ be as above. Then
  \[
    H^i(\tilde{X},\phi_\ast j_* \mathbb{G}_{m,K}) = 
    \begin{cases} 
      K^\times & \mbox{ if } i= 0 \\
      0  & \mbox{ if } i= 1 \\
      \Br_0 K & \mbox{ if } i=2 \\
      0  & \mbox{ if } i > 2 \\
    \end{cases}
  \]
where $\Br_0 K$ is the subgroup of the Brauer group of $K$ which has zero local invariants at the real infinite primes. For $\phi_\ast \DIV X$ we have
  \[
    H^i(\tilde{X},\phi_\ast \DIV X) = 
    \begin{cases} 
      \bigoplus_{x \in X} \ZZ  & \mbox{ if } i= 0 \\
      0  & \mbox{ if } i= 1 \\
      \bigoplus_{x \in X} \Br K_x & \mbox{ if } i=2 \\
      0  & \mbox{ if } i > 2 \\
    \end{cases}
  \]
where the direct sum ranges over the closed points in $X$ and $\Br K_x$ is the Brauer group of the completion of $K$ at $x$.
\end{prop}

We have a map
  \[
    \Gamma(\tilde{X},\HOM(\mathcal{E}_n,\phi_\ast \mathcal{C})) \rightarrow \RHom(\ZZ/n\ZZ,\phi_\ast \mathbb{G}_{m,X})\,,
  \] 
induced from the map $\Gamma \rightarrow \RGamma$. Since $\Gamma(\tilde{X},\HOM(\mathcal{E}_n,\phi_\ast \mathcal{C}))$ is $2$-truncated, this map will factor through $\tau^{\leq 2} ( \RHom(\ZZ/n\ZZ,\phi_\ast \mathbb{G}_{m,X}))$. We denote by 
  \[
    \psi \colon \Gamma(\tilde{X}, \HOM(\mathcal{E}_n,\phi_\ast \mathcal{C})) \rightarrow \tau^{\leq 2}( \RHom(\ZZ/n\ZZ,\phi_\ast \mathbb{G}_{m,X}))
  \] 
the resulting map. As in \cite[Lemma 4.2]{CarlsonSchlankUnramified}, $\psi$ is quickly shown to be an isomorphism. This allows one to compute the Ext-groups we are after. If $a \in K^\times$ we let $\divis(a)$ be the divisor in $\Div(X)$ determined by the corresponding fractional ideal in $\OO_K$.

\begin{cor}\label{cor:valuesofext}
Let $X = \Spec \OO_K$ for $K$ a number field. Then 
  \[
    \Ext^i_{\tilde{X}}(\ZZ/n\ZZ,\phi_\ast \mathbb{G}_{m,X}) =  
    \begin{cases} \mu_n(K) & \mbox{ if } i= 0 \\
      Z_1/B_1 & \mbox{ if } i= 1 \\
      \Pic X /n & \mbox{ if } i=2 \\
      \ZZ/n\ZZ & \mbox{ if } i=3 \\
      0 &  \mbox{ if } i>3\,.
    \end{cases}
  \] 
Here 
  \[
    Z_1 = \{(a,\mathfrak{a}) \in K^\times \oplus \Div X : \divis(a)+n\mathfrak{a} = 0\}
  \] 
and 
  \[
    B_1 = \{(b^{-n},\divis(b)) \in K^\times \oplus \Div X : b \in K^\times \}\,.
  \]
\end{cor}

\begin{rk}
  Note that we have $\Ext^i_X(\ZZ/n\ZZ,\GG_{m,X})\cong H^i(X_{\fl},\mmu_n)$ for all $i$, where $X_{\fl}$ denotes the big fppf site on $X$. This can be seen as follows: we have quasi-isomorphisms $(\ZZ \xrightarrow{n}\ZZ)\simeq \ZZ/n\ZZ$ and $\mmu_n \simeq (\GG_m \xrightarrow{n}\GG_m)$ of complexes of fppf sheaves and we get 
    $
      R\HOM(\ZZ/n\ZZ, \GG_m)\simeq \HOM(\ZZ \xrightarrow{n} \ZZ, \GG_m)\simeq (\GG_m \xrightarrow{n}\GG_m)\simeq \mmu_n   
    $  
  since $(\ZZ \xrightarrow{n}\ZZ)$ is a complex of free sheaves. 
  Applying $R\Gamma(X, -)$, we get that $R\Hom(\ZZ/n\ZZ, \GG_{m,X})\simeq R\Gamma(X, \mmu_n)$\,.
\end{rk}

The corollary above gives us a concrete description of $H^i(\tilde{X},\ZZ/n\ZZ)$ for all $i$. Indeed, to remind the reader, since $\phi_\ast \ZZ/n\ZZ$ is the constant sheaf on $\tilde{X}$ with value $\ZZ/n\ZZ$, Theorem \ref{thm:AVduality} applies. Thus, with $Z_1$ and $B_1$ as in Corollary \ref{cor:valuesofext}:
  \[
    H^i(\tilde{X},\ZZ/n\ZZ) = 
    \begin{cases} 
      \ZZ/n\ZZ & \mbox{ if } i= 0 \\
      (\Pic(X)/n)^\sim & \mbox{ if } i= 1 \\
      (Z_1/B_1)^\sim & \mbox{ if } i=2 \\
      \mu_n(K)^\sim & \mbox{ if } i=3 \\
      0 &  \mbox{ if } i>3.
    \end{cases}
  \]

\subsection{Galois coverings}
Let $G$ be a finite group. Any element $x$ in the pointed set $H^1(X,G)$ of right $G$-torsors in $\Xet$ can be represented by a Galois $G$-cover $f\colon Y \rightarrow X$. The goal of what remains in this section is to show that any element $x \in H^1(\Xcet,G)$ can be represented by a Galois $G$-cover in $\Xcet$. To make sense of this, we must of course define what a Galois covering should be in the category $\Xcet$. This is done in the following definition. A similar definition was made by Zink in \cite[Appendix 2, 2.6.1]{HaberlandGalois}.

\begin{df} \label{def:galoiscover}
Let $X = \Spec \OO_K$ be the ring of integers of a number field, and let $f \colon (Y,M) \rightarrow (X,X_\infty)$ be an object of $\Xcet$. Assume that the finite group $G$ acts on $f$ to the right. Then we say that $f$ is a Galois covering with Galois group $G$ if:
  \begin{enumerate}
    \item $f\colon Y \rightarrow X$ is a (not necessarily connected) Galois covering with Galois group $G$ in $\Xet$.
    \item The action of $G$ on $M$ is free, and $f$ induces an isomorphism $f\colon M/G \rightarrow X_\infty$.
  \end{enumerate}
\end{df}

\begin{rk} \label{rmk:galois}
Note that if $f\colon (Y,M) \rightarrow (X,X_\infty)$ is a Galois covering with Galois group $G$, then $M= Y_\infty$. Further, it is clear that every connected Galois covering with Galois group $G$ gives rise to a Galois extension $L/K$ with Galois group $G$ that is unramified at all places, including the infinite ones. Conversely, given a Galois extension $L$ of $K$ with Galois group $G$ that is unramified at the finite as well as at the infinite places, one gets a connected Galois covering with Galois group $G$.
\end{rk}

For any Galois $G$-covering $\tilde{Y} = (Y,Y_\infty) \rightarrow (X,X_\infty)$, the functor $h_{\tilde{Y}} = \Hom_{\Xcet}(-,\tilde{Y})$ gives rise to a (right) $G$-torsor in $\Sh(\Xcet)$. We thus have a map 
  \[
    c \colon \frac{\{ \text{Galois G-coverings of }(X,X_\infty) \}}{\sim} \rightarrow \frac{\{\text{G-torsors on } \Xcet\}}{\sim}
  \] 
where $\sim$ means we are passing to isomorphism classes. The following lemma is a straightforward application of descent theory (see for example \cite[Theorem 4.3]{MilneEtale}).

\begin{lem} \label{lemma:torsorGalois}
Let $G$ be a finite group. Then the map $c$ defined above is a bijection.
\end{lem}

If $G$ is abelian, then the set of isomorphism classes of (right) $G$-torsors has the structure of an abelian group. Indeed, if $\mathcal{F}_1,\mathcal{F}_2$ are $G$-torsors, define 
  \[
    \mathcal{F}_1 \wedge^G \mathcal{F}_2 = (\mathcal{F}_1 \times \mathcal{F}_2)/G
  \] 
where $g \in G$ acts by taking $(x,y)$ to $(xg^{-1},yg)$. The operation $\wedge^G$ respects isomorphism classes and descends to an operation on the set of isomorphism classes of $G$-torsors. One can then verify that $\wedge^G$ gives the set of isomorphism classes of (right) $G$-torsors the structure of an abelian group with unit the trivial $G$-torsor. In any Grothendieck topos $\mathcal{T}$, we have that if $G$ is an abelian group, then the group $H^1(\mathcal{T},G)$ is isomorphic to the isomorphism classes of $G$-torsors on $\mathcal{T}$. Now lemma \ref{lemma:torsorGalois} and the fact that $\Sh(\Xcet)$ is a Grothendieck topos gives the following result.

\begin{lem} \label{lemma:torsorGalois2}
Let $X = \Spec \OO_K$ be the ring of integers of a number field. Then if $G$ is the constant sheaf associated to a finite group, any element $x \in H^1(\tilde{X},G)$ can be represented by a Galois cover $(Y,Y_\infty) \rightarrow (X,X_\infty)$ in $\Xcet$ with Galois group $G$.
\end{lem}

In Remark \ref{rmk:galois}, we noted that any \emph{connected} Galois covering $p \colon (Y,Y_\infty) \rightarrow (X,X_\infty)$ with finite Galois group $G$ was represented by an unramified extension of $K$. If $Y$ is not connected, we now show that $Y$ is isomorphic to a covering that is induced from a subgroup of $G$. If $H \subset G$ is a subgroup of $G$ and $q\colon (Z,Z_\infty) \rightarrow (X,X_\infty)$ is a Galois $H$-covering, we can form the induced cover $(Y,Y_\infty) := \Ind^G_H((Z,Z_\infty)) \rightarrow (X,X_\infty)$ as follows. Let $n=[G:H]$ and choose a set $\{g_1, \ldots, g_n\}$ of right coset representatives of $H$ in $G$. For each $g_i$, we denote by $(Z,Z_\infty)g_i$ a copy of $(Z,Z_\infty)$, which should be seen as marked by $g_i$. We then define the map 
  \[
    p\colon \Ind^G_H((Z,Z_\infty)) = \amalg_{i=1}^n (Z,Z_\infty)g_i \rightarrow (X,X_\infty)
  \] 
on the component $(Z,Z_\infty)g_i$ as just the map $q\colon (Z,Z_\infty) \rightarrow (X,X_\infty)$. We now define the action of $G$ on $\Ind^G_H(Z) = \amalg_{i=1}^n (Z,Z_\infty)g_i$. If $x \in (Z,Z_\infty)g_i$, $g \in G$, then there exists a unique coset representative $g_j$ such that $g_ig = hg_j$ for some $h \in H$. We then let 
  \[
    xg := xh \in (Z,Z_\infty)g_{j}\,.
  \]   
This is just the Galois-covering analogue of an induced representation. The induced coverings generate all Galois coverings in the sense of the following Lemma.

\begin{lem}\label{lemma:torsorinduced}
Let $G$ be a finite group and $H \subset G$ a subgroup. Then for any Galois $H$-cover $q\colon (Z,Z_\infty) \rightarrow (X,X_\infty)$, the induced cover
  \[
    \Ind^G_H((Z,Z_\infty)) \rightarrow (X,X_\infty)
  \] 
is a Galois $G$-cover. Conversely, any Galois $G$-cover $p\colon (Y,Y_\infty) \rightarrow (X,X_\infty)$ is isomorphic to a Galois $G$-covering $\Ind^G_H((Z,Z_\infty))$ for a subgroup $H \subset G$ unique up to conjugation, and a connected Galois $H$-covering $(Z,Z_\infty) \rightarrow (X,X_\infty)$, unique up to isomorphism.
\end{lem}

\begin{proof}
This is standard, so we will indicate the proof and leave some of the details to the reader. It is clear that $\Ind^G_H((Z,Z_\infty))$ is a Galois $G$-cover. To see that any Galois $G$-cover arises in this way for a unique subgroup $H \subset G$ and a unique connected Galois $H$-cover $q\colon (Z,Z_\infty) \rightarrow (X,X_\infty)$, let $p\colon (Y,Y_\infty) \rightarrow (X,X_\infty)$ be a Galois $G$-cover. If $Y$ is connected, the statement is trivial, since we then can take $H = G$ and $(Z,Z_\infty) = (Y,Y_\infty)$. Thus we assume that $Y$ is not connected, say $Y_1, \ldots, Y_n$ are its components. We let $H \subset G$ consist of those $g \in G$ such that $(Y_1,(Y_1)_\infty)g \subset (Y_1,(Y_1)_\infty).$ Choose $g_1, \ldots, g_n$ such that $((Y_1,(Y_1)_\infty))g_i= (Y_i,(Y_i)_\infty)$. We then see that the $(Y_i,(Y_i)_\infty)$ are isomorphic to each other. Further, $(Y_1,(Y_1)_\infty)$ is a Galois $H$-cover since $Y$ is a Galois $G$-cover, and it is clear that 
  \[
    (Y,Y_\infty)  \cong \Ind^G_H((Y_1,(Y_1)_\infty))\,.
  \] 
The unicity claims are left to the reader.
\end{proof}

Note that Lemma \ref{lemma:torsorinduced} has the consequence that any 
  \[
    x \in H^1(\tilde{X},\ZZ/n\ZZ) \cong (\Cl K /n \Cl K)^\sim
  \] 
can be represented by an unramified cyclic extension $L/K$ of degree $d$ dividing $n$, together with a choice of generator $\sigma\in \Gal(L/K)$, in the sense that $x$ can be represented by a Galois $\ZZ/n\ZZ$-covering of the form 
  \[
    \Ind^{\ZZ/n\ZZ}_{\ZZ/d\ZZ}((Y,Y_\infty)) \rightarrow  (X,X_\infty)\,,
  \] 
where $Y = \Spec \OO_L$. 

We now move on to the last lemma of the section. Let $p \colon \tilde{Y} \rightarrow \tilde{X}$ be a Galois cover of $\tilde{X}$ with Galois group $G$. If $F$ is an abelian sheaf on $\Xcet$, we say that $F$ is $p$-split if $p^*F$ is a constant sheaf on $\Ycet$. There is a natural action of $G$ on $p^*F$ and in this manner we get a functor from the category of sheaves split by $p$ to the category of $G$-modules. The proof of the following lemma follows once again from standard descent theory.

\begin{lem} \label{lemma:locconst}
Let $p \colon \tilde{Y} \rightarrow \tilde{X}$ be a Galois cover of $\tilde{X}$ with Galois group $G.$ Then the category of $p$-split abelian sheaves on $\Xcet$ is equivalent to the category of left $G$-modules.
\end{lem}

\begin{proof}
  See \cite[p. 532]{MazurNotes}. 
\end{proof}


\section{The cohomology ring of a number field} \label{sec:cup}
The aim of this section is to compute the cohomology ring $H^*(\tilde{X},\ZZ/n\ZZ)$ for $X = \Spec \OO_K$ the ring of integers of a number field. In \cite{CarlsonSchlankUnramified}, this was done when $K$ is totally imaginary and $n=2$; if $K$ is a number field that is not totally imaginary, or if $n>2$, the methods utilized in that paper have to be altered. To compute the cohomology ring, we note that the fact that $H^i(\tilde{X},\ZZ/n\ZZ) = 0 $ for $i>3$, together with graded commutativity of the cup product, shows that it is enough to calculate $x \cup y$,  where $y \in H^i(\tilde{X},\ZZ/n\ZZ),i=0,1,2$ and $x \in H^1(\tilde{X},\ZZ/n\ZZ)$.  For $i=0$ the result is obvious, so we are reduced to when $i=1,2$. We denote by $c_x$ the map 
  \[
    x \cup -  \colon H^i(\tilde{X},\ZZ/n\ZZ) \rightarrow H^{i+1}(\tilde{X},\ZZ/n\ZZ)\,.
  \]  

First of all, we should be precise with what we mean when we say that we compute $c_x$. Let us note that by Lemma \ref{lemma:torsorGalois2} and Lemma \ref{lemma:torsorinduced}, $x$ can be represented by a Galois covering $\Ind^{\ZZ/n\ZZ}_{\ZZ/d\ZZ}((Y,Y_\infty))$ for $Y = \Spec \OO_L$ the ring of integers of a cyclic extension $L/K$ of degree $d|n$ which is unramified at the finite as well as the infinite places. In other words, $x$ is represented by $L$, together with a choice of generator $\sigma\in \Gal(L/K)$. By Lemma \ref{lemma:morphismdual}, the map $c_x$ is, under Artin--Verdier duality, dual to a map 
  \[
    c_x^\sim \colon \Ext^{3-(i+1)}_{\tilde{X}}(\ZZ/n\ZZ_{\tilde{X}},\phi_*\mathbb{G}_{m,X}) \rightarrow \Ext^{3-i}_{\tilde{X}}(\ZZ/n\ZZ_{\tilde{X}},\phi_* \mathbb{G}_{m,X})\,.
  \] 
We will compute the map $c_x^\sim$ under the identifications of Corollary \ref{cor:valuesofext}.
We start by showing that the map $c_x\colon H^i(\tilde{X},\ZZ/n\ZZ) \rightarrow H^{i+1}(\tilde{X},\ZZ/n\ZZ)$ can be identified with a connecting homomorphism coming from a certain exact sequence of sheaves on $\Xcet$. By Lemma \ref{lemma:torsorGalois2}, we can represent $x \in H^1(\tilde{X},\ZZ/n\ZZ)$ by a Galois covering 
  \[
    p\colon \tilde{Y} \to \tilde{X}\,.
  \]
Since $p$ is finite \'etale, we have by Proposition \ref{prop:adjoint} that $p_*$ is also left adjoint to $p^*$, and the counit 
  \[
    N\colon p_*p^* \ZZ/n\ZZ_{\tilde{X}} \rightarrow \ZZ/n\ZZ_{\tilde{X}}
  \] 
will be called the \emph{norm}. This terminology is bad from the point of view of algebraic geometry where it is usually called \emph{trace} but it agrees well with the number theoretic terminology and hence we will stick with \emph{norm}. The norm $N$ is an epimorphism, which can be seen by pulling back sections along the cover $p\colon Y\to X$.  

This gives us a short exact sequence
  \begin{equation}\label{eq:1}
    0\to \ker N \xrightarrow{u} p_*p^*\ZZ/n\ZZ_{\tilde{X}} \xrightarrow{N} \ZZ/n\ZZ_{\tilde{X}} \to 0\,.
  \end{equation}
Denote by $C_n=\ZZ/n\ZZ$. Let us note that under the equivalence between the category of sheaves split by the morphism $p$ and the category of $C_n$-sets given by Lemma \ref{lemma:locconst}, $p_*p^* \ZZ/n\ZZ_{\tilde{X}}$ corresponds to the left $C_n$-module which is the group ring 
  \[
    \ZZ/n\ZZ[C_n] \cong \ZZ/n\ZZ[e]/(e^n-1)\,.
  \]
One also easily shows that the norm map 
  \[
    N\colon p_*p^* \ZZ/n\ZZ_{\tilde{X}} \rightarrow \ZZ/n\ZZ_{\tilde{X}}
  \] 
corresponds, under the given equivalence of categories, to the augmentation map 
  \[
    \epsilon \colon \ZZ/n\ZZ[C_n] \rightarrow \ZZ/n\ZZ\,.
  \] 
Here $\ZZ/n\ZZ$ has the trivial $C_n$-action, and the augmentation map is the map that takes $g \in C_n$ to $1$. This gives us that $\ker N$ corresponds to $\ker \epsilon$ under the stated equivalence of categories. Thus, exact sequence (\ref{eq:1}) corresponds, under the equivalence in Lemma \ref{lemma:locconst} to the exact sequence
  \[
    0 \rightarrow \ker \epsilon \rightarrow \ZZ/n\ZZ[C_n] \xrightarrow{\epsilon} \ZZ/n\ZZ \rightarrow 0\,.
  \] 
It is easy to see that $\ker \epsilon$, as a $\ZZ/n\ZZ$-module, is free on the elements $g-1$, for $0\neq g \in C_n$. There is a $C_n$-equivariant map
  \[
    s \colon \ker \epsilon \rightarrow \ZZ/n\ZZ
  \] 
taking $e-1$ to $1$. If we let $P$ be the pushout of 
  \[
    \ZZ/n\ZZ[C_n] \leftarrow \ker \epsilon \xrightarrow{s} \ZZ/n\ZZ\,,
  \] 
we have the commutative diagram 
  \[
    \begin{tikzcd}  
      0  \arrow[r] & \ker \epsilon \arrow[d,"s"]  \arrow[r] &  \ZZ/n\ZZ[C_n] \arrow[r,"\epsilon"]  \arrow[d] &  \ZZ/n\ZZ \arrow[d,equals] \arrow[r] & 0  \\ 
      0 \arrow[r] & \ZZ/n\ZZ \arrow[r] & P \arrow[r] & \ZZ/n\ZZ \arrow[r] & 0 \,. 
    \end{tikzcd}
  \] 
The short exact sequence 
  \begin{equation} \label{ses} 
    \begin{tikzcd} 
      0 \arrow[r] & \ZZ/n\ZZ \arrow[r] & P \arrow[r] & \ZZ/n\ZZ \arrow[r] & 0 
    \end{tikzcd}
  \end{equation} 
corresponds under the equivalence given by Lemma \ref{lemma:locconst} to a short exact sequence 
  \begin{equation} \label{ses2} 
    0 \rightarrow \ZZ/n\ZZ_{\tilde{X}} \rightarrow \mathcal{F} \rightarrow \ZZ/n\ZZ_{\tilde{X}} \rightarrow 0
  \end{equation} 
of sheaves on $\Xcet$ and we have a commutative diagram
  \begin{equation} \label{comdiag} 
    \begin{tikzcd}   
      0  \arrow[r] & \ker N \arrow[d,"f"]  \arrow[r,"u"] &  p_*p^*  \ZZ/n\ZZ_{\tilde{X}}  \arrow[r,"N"]  \arrow[d] &   \ZZ/n\ZZ_{\tilde{X}}  \arrow[d,equals] \arrow[r] & 0  \\ 
      0 \arrow[r] & \ \ZZ/n\ZZ_{\tilde{X}} \arrow[r] & \mathcal{F} \arrow[r] &  \ZZ/n\ZZ_{\tilde{X}}  \arrow[r] & 0 \,.
    \end{tikzcd} 
  \end{equation}

\begin{lem} \label{lemma:transfercup}
Let $x \in H^1(\tilde{X},\ZZ/n\ZZ)$ be represented by the $\ZZ/n\ZZ$-torsor $p\colon \tilde{Y} \rightarrow \tilde{X}$. Then the connecting homomorphism $H^i(\tilde{X},\ZZ/n\ZZ) \rightarrow H^{i+1}(\tilde{X},\ZZ/n\ZZ)$ arising from the short exact sequence (\ref{ses2}) is given by cup product with $x \in H^1(\tilde{X},\ZZ/n\ZZ)$.
\end{lem}

\begin{proof}
The strategy is the same as in \cite[Lemma 5.1]{CarlsonSchlankUnramified}. Note that under the canonical identification $H^1(\tilde{X},\ZZ/n\ZZ)\cong \Ext^1_{\tilde{X}}(\ZZ/n\ZZ, \ZZ/n\ZZ)$, the cup product corresponds to the Yoneda product. Now let $C_n=\ZZ/n\ZZ$. By \cite[VIII.2, Theorem 7]{MaclaneTopos}, $x$, viewed as a $\ZZ/n\ZZ$-torsor, corresponds to a geometric morphism 
  \[
    k_x \colon \Sh(\Xcet) \rightarrow BC_n
  \] 
(i.e., an adjunction where the left adjoint preserves all finite limits) where $BC_n$ is the topos of $C_n$-sets. In this topos, there is a universal $\ZZ/n\ZZ$-torsor, which we denote $U_{\ZZ/n\ZZ}$. The underlying $C_n$-set of $U_{\ZZ/n\ZZ}$ is just $\ZZ/n\ZZ$ with $C_n$ acting by left translation, and $\ZZ/n\ZZ$ acts on $U_{\ZZ/n\ZZ}$ by right translation. Given a $\ZZ/n\ZZ$-torsor $T$, we define $\mu(T) \in \Ext^1_{C_n}(\ZZ/n\ZZ,\ZZ/n\ZZ)$ as follows. We let $\ZZ/n\ZZ[T]$ be the $C_n$-module whose elements are given by formal sums $\sum_i a_i [t_i]$, $a_i \in \ZZ/n\ZZ$, $t_i \in T$  and where $C_n$ acts in the obvious way. There is a map $\epsilon_{T}  \colon \ZZ/n\ZZ[T] \rightarrow \ZZ/n\ZZ$ given by mapping $\sum_i a_i [t_i] \to \sum_i a_i$. The kernel $\ker \epsilon_T$ is generated by elements of the form $[t_1]-[t_2], t_1,t_2 \in T$ and we define a map $f_T \colon \ker \epsilon \rightarrow \ZZ/n\ZZ$ by mapping $[t_1]-[t_2]$ to the unique $g \in \ZZ/n\ZZ \cong C_n$ such that $gt_2 = t_1$. We then define $\mu(T) \in \Ext^1_{C_n}(\ZZ/n\ZZ,\ZZ/n\ZZ)$ to be the short exact sequence we get by pushout along $f_T$ of the exact sequence
  \[
    0 \rightarrow \ker \epsilon_T \rightarrow \ZZ/n\ZZ[T] \rightarrow \ZZ/n\ZZ \rightarrow 0\,.
  \]
It is well-known that this gives an isomorphism $\Tors_{C_n}(\ZZ/n\ZZ) \rightarrow \Ext^1_{C_n}(\ZZ/n\ZZ,\ZZ/n\ZZ)$. We then see that $\mu(U_{\ZZ/n\ZZ})$ corresponds to a short exact sequence 
  \begin{equation} \label{ses3} 
    0 \rightarrow \ZZ/n\ZZ \rightarrow P \rightarrow \ZZ/n\ZZ \rightarrow 0\,.
  \end{equation}  
If we take cohomology, then the connecting homomorphism $\Ext^i_{C_n}(\ZZ/n\ZZ,\ZZ/n\ZZ) \rightarrow \Ext^{i+1}_{C_n}(\ZZ/n\ZZ,\ZZ/n\ZZ)$ is given by the Yoneda product with $\mu(U_{\ZZ/n\ZZ})$. If we pull-back the short exact sequence (\ref{ses3}) by $k_x^*$ we get the short exact sequence (\ref{ses2}), and our claim now follows. Indeed, the connecting homomorphism from short exact sequence (\ref{ses2}) is given by Yoneda product with the element in $\Ext^1_{\tilde{X}}(\ZZ/n\ZZ_{\tilde{X}},\ZZ/n\ZZ_{\tilde{X}})$ classifying it. Thus the connecting homomorphism is given by cup product with $x$.
\end{proof}

The commutative diagram (\ref{comdiag}) shows that if $\delta_x$ is the connecting homomorphism coming from the upper short exact sequence, then
the diagram  
  \[
    \xymatrix{H^i(\tilde{X},\mathbb{Z}/n\ZZ)\ar[r]^{\delta_x} \ar[d]^= & H^{i+1}(\tilde{X},\ker N)\ar[d]^{f_*} \\ H^i(\tilde{X},\mathbb{Z}/n\ZZ)\ar[r]^{c_x} & H^{i+1}(\tilde{X},\mathbb{Z}/n\ZZ)}
  \] 
commutes. Our plan to compute the cup product is now to first compute the map $\delta_x$, and then to compose with $f_*$. By the above commutative diagram, this agrees with the cup product map. By applying Lemma \ref{lemma:morphismdual} to the map
  \[
    \delta_x\colon \ZZ/n\ZZ_{\tilde{X}} \rightarrow \ker N[1]\,,
  \] 
we see that the map
  \[
    \RGamma(\tilde{X},\delta_x)\colon \RGamma(\tilde{X},\ZZ/n\ZZ) \rightarrow \RGamma(\tilde{X},\ker N[1])
  \]
corresponds under Artin--Verdier duality to
  \[
    \RHom_{\tilde{X}}(\delta_x,\phi_\ast \mathbb{G}_{m,X})\colon \RHom_{\tilde{X}}(\ker N [1],\phi_\ast \mathbb{G}_{m,X})[3] \rightarrow \RHom_{\tilde{X}}(\ZZ/n\ZZ_{\tilde{X}},\phi_\ast \mathbb{G}_{m,X})[3]\,.
  \]  
This shows that $\delta_x$ is Pontryagin dual to the map 
  \[
    \delta_x^\sim\colon \Ext_{\tilde{X}}^{3-(i+1)} (\ker N, \phi_\ast \mathbb{G}_{m,X}) \rightarrow \Ext_{\tilde{X}}^{3-i} (\ZZ/n\ZZ_{\tilde{X}}, \phi_\ast \mathbb{G}_{m,X})
  \] 
which we get by applying $H^{3-i}$ to $\RHom_{\tilde{X}}(\delta_x,\phi_* \mathbb{G}_{m,X})$. In the same way we see that the map $c_x\colon H^i(\tilde{X},\ZZ/n\ZZ) \rightarrow H^{i+1}(\tilde{X},\ZZ/n\ZZ)$ is, under Artin--Verdier duality, Pontryagin dual to the map $c_x^\sim$, which is the composite 
  \[
    \Ext^{3-(i+1)}_{\tilde{X}}(\ZZ/n\ZZ_{\tilde{X}}, \phi_* \mathbb{G}_{m,X}) \xrightarrow{f^*} \Ext^{3-(i+1)}_{\tilde{X}}(\ker N ,\phi_* \mathbb{G}_{m,X}) \xrightarrow{\delta_x^\sim} \Ext^{3-i}_{\tilde{X}}(\ZZ/n\ZZ_{\tilde{X}},\phi_*\mathbb{G}_{m,X}). 
  \]
We will now compute $\delta_x^\sim$ and $c_x^\sim$ by taking resolutions of $\ker N,\ \phi_\ast \mathbb{G}_{m,X}$ and $\ZZ/n\ZZ$.
Since, under the equivalence between locally constant sheaves split by $p$ and $C_n$-sets, $\ker N$ corresponds to $\ker \epsilon$, to resolve $\ker N$, it is enough to find a resolution of $\ker \epsilon$, and that is what we will do.  Let us denote the element $\sum_{g \in C_n} g$ by $\Delta$ and choose a generator $e$ of $C_n$, thus establishing an isomorphism $\ZZ[C_n] \cong \ZZ[e]/(e^n-1)$. The resolution we will use is the following 
  \[
    \mathcal{K} = ( \mathbb{Z} \xrightarrow{\begin{pmatrix}n \\ -\Delta\end{pmatrix}}  \mathbb{Z}\oplus \mathbb{Z}[C_n] \xrightarrow{\begin{pmatrix}\Delta & n\end{pmatrix}} \mathbb{Z}[C_n] )
  \] 
that is, the first map is multiplication by $n$ on the first factor and multiplication by $-\Delta$ on the second factor. The second map is multiplication by $\Delta$ on the first factor and by $n$ on the second factor. The map $\mathcal{K} \rightarrow \ker \epsilon$ taking $1 \in \ZZ[C_n]$ to $e-1$ then exhibits $\mathcal{K}$ as a resolution of $\ker \epsilon$. By the equivalence of categories between $C_n$-sets and locally constant sheaves split by $p$, we get the resolution
  \[
    \mathbb{Z}_{\tilde{X}} \xrightarrow{d_2} \mathbb{Z}_{\tilde{X}} \oplus p_*p^*\mathbb{Z}_{\tilde{X}} \xrightarrow{d_1} p_*p^*\mathbb{Z}_{\tilde{X}} 
  \]  
of $\ker N$. We will by abuse of notation also denote by $\K$, the complex resolving $\ker N$. Let us now resolve $\phi_\ast  \mathbb{G}_{m,X}$ as in Section \ref{subsec:artin}, by the complex $\mathcal{C}$, defined as 
  \[
    \phi_\ast j_\ast \mathbb{G}_{m,K} \xrightarrow{\phi_\ast \divis} \phi_\ast \DIV X \rightarrow 0\,.
  \]
Note that $\phi_\ast \divis$ is an epimorphism since $\divis\colon j_*\mathbb{G}_{m,K}\to \DIV X$ is an epimorphism and $j^*\DIV X = 0$. Let $\mathcal{E}_n$ be the complex 
  \[
    \mathbb{Z}_{\tilde{X}} \xrightarrow{n} \mathbb{Z}_{\tilde{X}} 
  \]  
resolving $\ZZ/n\ZZ_{\tilde{X}}$. We will now show that the maps $u \colon \ker N \rightarrow p_*p^* \ZZ/n\ZZ_{\tilde{X}},\ N \colon p_*p^* \ZZ/n\ZZ_{\tilde{X}} \rightarrow \ZZ/n\ZZ_{\tilde{X}}$, and $f \colon \ker N \rightarrow \ZZ/n\ZZ_{\tilde{X}}$ from commutative diagram (\ref{comdiag}) lift to morphisms of complexes $\hat{u} \colon \mathcal{K} \rightarrow p_*p^* \mathcal{E}_n$, $\widehat{N}\colon p_*p^* \mathcal{E}_n \rightarrow \mathcal{E}_n$ and $\hat{f}\colon \mathcal{K} \rightarrow \mathcal{E}_n$. We will explain how these morphisms are defined for the corresponding $C_n$-sets, using once again the equivalence between locally constant sheaves split by $p$ and $C_n$-sets. The map $\hat{u}$ is defined as:
  \[
    \begin{tikzcd}[column sep = huge, ampersand replacement=\&] 
      \mathbb{Z}  \ar{r}{\begin{pmatrix} n \\ -\Delta \end{pmatrix}} \ar{d} \& \mathbb{Z} \oplus \mathbb{Z}[C_n] \ar{d}{\begin{pmatrix}0 & e-1\end{pmatrix}} \ar{r}{\begin{pmatrix}\Delta & n\end{pmatrix}} \& \ZZ[C_n] \ar{d}{e-1} \\  0 \ar{r} \& \ZZ[C_n] \ar{r}{n} \& \ZZ[C_n],
    \end{tikzcd}
  \]
while the map $\widehat{N}$ is given by
  \[
    \begin{tikzcd}[column sep = huge ] 
      \mathbb{Z}[C_n]  \arrow[r,"n"] \arrow[d,"\epsilon"] & \mathbb{Z}[C_n] \arrow[d,"\epsilon"] \\ \ZZ \arrow[r,"n"] & \ZZ .
    \end{tikzcd}
  \]
Lastly, the map $\hat{f}$ is given in components as 
  \[
    \begin{tikzcd}[column sep = huge, ampersand replacement=\&] 
      \mathbb{Z}\ar{r}{\begin{pmatrix} 
        n \\
        - \Delta
      \end{pmatrix}} \ar{d} \& \mathbb{Z} \oplus \mathbb{Z}[C_n] 
      \ar{d}{\begin{pmatrix}
        1 & \epsilon
      \end{pmatrix}} \ar{r}{\begin{pmatrix}\Delta & n\end{pmatrix}} \& \ZZ[C_n] \ar{d}{\epsilon} \\ 
      0 \ar{r} \& \ZZ \ar{r}{n} \& \ZZ,
    \end{tikzcd}
  \]
Let $\Cone(\hat{u}) = p_*p^* \mathcal{E}_n \oplus \mathcal{K}[1]$ denote the cone of $\hat{u}.$ Since $\widehat{N} \circ \hat{u} = 0$, we have a map $q(\hat{u})  \colon  \Cone(\hat{u}) \xrightarrow{(\widehat{N},0)} \mathcal{E}_n$.

Note that the map $q(\hat{u})\colon \Cone(\hat{u}) \rightarrow \mathcal{E}_n$ is a quasi-isomorphism (see e.g. \cite[III.3.5]{Gelfand--Manin}).
Summarizing our situation, we have the zig-zag \[\mathcal{E}_n \xleftarrow{q(\hat{u})} \Cone(\hat{u})\xrightarrow{\pr_2} \mathcal{K}[1] \] and a map $\hat{f}[1] \colon \mathcal{K}[1] \rightarrow \mathcal{E}_n[1]$. The zig-zag represents $\delta_x \colon \ZZ/n\ZZ_{\tilde{X}} \rightarrow \ker N[1].$ We now apply $\HOM(-,\mathcal{C})$ to this zig-zag and $\hat{f}$ to get the zig-zag
  \begin{equation}\label{eq:zigzag}
    \HOM(\mathcal{E}_n,\mathcal{C}) \xrightarrow{q(\hat{u})^*} \HOM(\Cone(\hat{u}),\mathcal{C}) \xleftarrow{\pr_2^*} \HOM(\mathcal{K}[1],\mathcal{C})
  \end{equation}  
and the map 
  \[
    \hat{f}^* \colon \HOM(\mathcal{E}_n[1],\mathcal{C}) \rightarrow \HOM(\mathcal{K}[1],\mathcal{C})\,.
  \]
The map $q(\hat{u})^*$ is a quasi-isomorphism since $q(\hat{u})$ is a quasi-isomorphism between complexes of locally free sheaves. Applying the global sections functor to the zig-zag (\ref{eq:zigzag}) and using the natural transformation $\Gamma \rightarrow \RGamma$, we have the commutative diagram 
  \[
    \begin{tikzcd}[column sep= huge] 
      \Hom(\mathcal{E}_n,\mathcal{C})  \arrow[r,"{q(\hat{u})^*}"]  \arrow[d,"s"]& \Hom(\Cone(\hat{u}),\mathcal{C})  \arrow[d,"t"] &   \arrow[l,swap,"{\pr_2^*}"] \Hom(\mathcal{K}[1],\mathcal{C}) \arrow[d] \\
      R\Hom(\mathcal{E}_n,\mathcal{C}) \arrow[r,"{\RHom(q(\hat{u}),\mathcal{C})}"] & \RHom(\Cone(\hat{u}),\mathcal{C})  &  \arrow[l,swap,"{\RHom(\pr_2,\mathcal{C})}"]  \RHom(\mathcal{K}[1],\mathcal{C})\,.
    \end{tikzcd}
  \]
We want to prove the following lemma.

\begin{lem} \label{lemma:isoh}
The maps $s,t$ and $q(\hat{u})^*$ induce isomorphisms on $H^i$ for $i=0,1,2$.
\end{lem}

Before proving this lemma, we write out the complexes and the maps appearing in the zig-zag   
  \[
    \Hom(\mathcal{E}_n,\mathcal{C}) \xrightarrow{q(\hat{u})^*} \Hom(\Cone(\hat{u}), \mathcal{C}) \xleftarrow{\pr_2^*} \Hom(\mathcal{K}[1],\mathcal{C}) \xleftarrow{\hat{f}^*} \Hom(\mathcal{E}_n[1],\mathcal{C}) 
  \] 
explicitly. To do this, we first need some notation. Since $p \colon \tilde{Y} \rightarrow \tilde{X}$ is a Galois covering, it can be written as $\Ind^{C_n}_{C_d}(\tilde{Z})$ where 
  \[
    \tilde{Z} = (Z, Z_\infty) = (\Spec \OO_L, (\Spec \OO_L)_\infty)
  \]  
for $L$ an unramified field extension of $K$ of degree $d|n$ with Galois group $C_d$. Note that $\tilde{Y}$ has $n/d$ components. We let $\sigma$ be the fixed generator of $\Gal(L/K)$ corresponding to the element $x\in H^1(X, \ZZ/n\ZZ)$ that we started with and denote by $\sigma'$ a choice of generator of $\Gal(Y/X)$ such that the inclusion $\Gal(L/K)\subset \Gal(Y/X)$ takes $\sigma$ to $\sigma'^{n/d}$. We then get a set of right coset representatives $\{e, \sigma', \ldots ,\sigma'^{n/d-1}\}$ of $\Gal(L/K) \subset \Gal(Y/X)$. Using this set of coset representatives, we fix an isomorphism $\tilde{Y} \cong \amalg_{i=1}^{n/d} \tilde{Z}$. We write 
  \[
    \sigma'-1\colon (L^\times)^{n/d} \rightarrow (L^\times)^{n/d}\quad \mbox{and}\quad \sigma'-1\colon (\Div L)^{n/d} \rightarrow (\Div L)^{n/d}
  \] 
to denote the maps taking $a = (a_1,a_2, \ldots ,a_{n/d}) \in (L^\times)^{n/d}$ to 
  \[
    \sigma'(a)/a := (\sigma(a_{n/d})/a_1,a_1/a_2 \ldots , a_{n/d-1}/a_{n/d})\,,
  \]  
and 
  $
    I = (I_1, \ldots , I_{n/d}) \in \Div(Y)
  $ 
to 
  \[
    \sigma'(I)-I := (\sigma(I_{n/d})-I_1, I_1-I_2, \ldots, I_{n/d-1}-I_{n/d})
  \] 
respectively. There are also norm maps 
  $
    N_{Y|X} \colon (L^\times)^{n/d} \rightarrow K^\times
  $ 
and 
  $
    N_{Y|X} \colon (\Div L)^{n/d} \rightarrow \Div K
  $ 
taking $a = (a_1, \ldots, a_{n/d})  \in (L^\times)^{n/d}$ and $I = (I_1, \ldots, I_{n/d}) \in (\Div L)^{n/d}$ to 
  \[
    N_{Y|X}(a) = \Pi_{j=1}^{n/d} (\Pi_{i=0}^{d-1} \sigma^i (a_j))
    \quad\mbox{and}\quad
    N_{Y|X}(I) = \sum_{j=1}^{n/d} (\sum_{i=0}^{d-1} \sigma^i(I_j))
  \] 
respectively. We also have the obvious diagonal inclusion maps 
  $
    i\colon K^\times \rightarrow (L^\times)^{n/d}
  $ 
and
  $
    i\colon \Div K \rightarrow (\Div L)^{n/d}\,.
  $
Lastly, we have the maps $\divis \colon (L^\times)^{n/d} \rightarrow (\Div L)^{n/d}$ taking a tuple 
  $
    a= (a_1, \ldots, a_n) \in (L^\times)^{n/d}
  $
to 
  \[
    \divis(a) = (\divis(a_1), \ldots, \divis(a_n)) \in (\Div L)^{n/d}\,,
  \] 
where $\divis(a_i)$ is the fractional ideal of $L$ generated by $a_i$. If we have a complex $\mathcal{G}$, we will in the diagram that follows write $\mathcal{G}^*$ to denote $\HOM(\mathcal{G},\mathcal{C})$. All the maps and differentials we have are then collected in the diagram
  \begin{equation*} \label{diag} 
    \begin{tikzcd}[column sep = tiny] 
      K^\times \arrow[r,"{q(\hat{u})^*_0}"] \arrow[d,"d^0_{\mathcal{E}_n^*}"] & (L^\times)^{n/d} \arrow[d,"{d^0_{\Cone(\hat{u})^*}}"] & & \\  
      K^\times \oplus \Div K  \arrow[r,"{q(\hat{u})^*_1}"]  \arrow[d,"d^1_{\mathcal{E}_n^*}"] &  (\Div L)^{n/d} \oplus (L^\times)^{n/d} \oplus (L^\times)^{n/d}  \arrow[d,"{d^1_{\Cone(\hat{u})^*}}"]  &   \arrow[l,"{(\pr_2^*)_1}",swap]  \arrow[d,"{d^1_{\mathcal{K}[1]^*}}"] (L^\times)^{n/d} &\arrow[l,"\hat{f}^*_1",swap]  \arrow[d,"{d^1_{\mathcal{E}_n[1]^*}}"] K^\times \\  
      \Div K  \arrow[r,"{q(\hat{u})^*_2}"] & (\Div L)^{n/d} \oplus (\Div L)^{n/d} \oplus K^\times \oplus (L^\times)^{n/d} \arrow[d,"{d^2_{\Cone(\hat{u})^*}}"]  &   \arrow[l,"{(\pr_2^*)_2}"] \arrow[d,"{d^2_{\mathcal{K}[1]^*}}"]  (\Div L)^{n/d} \oplus K^\times \oplus (L^\times)^{n/d} & \arrow[l,"\hat{f}^*_2",swap] \arrow[d,"{d^2_{\mathcal{E}_n[1]^*}}"]  K^\times \oplus \Div K \\ 
      & \Div K \oplus (\Div L)^{n/d} \oplus K^\times  \arrow[d,"{d^3_{\Cone(\hat{u})^*}}"]  &  \arrow[l,"(\pr_2^*)_3",swap]  \arrow[d,"{d^3_{\mathcal{K}[1]^*}}"]  \Div K \oplus (\Div L)^{n/d} \oplus K^\times &  \arrow[l,"\hat{f}^*_3",swap]  \Div K \\ 
      & \Div K & \arrow[l,"(\pr_2^*)_4",swap] \Div K &
  \end{tikzcd} 
\end{equation*}
Where, if we write the maps in matrix form, the differentials are as follows 
  \[
    d^0_{\mathcal{E}_n^*} =
      \begin{pmatrix} -n \\ 
        \divis \end{pmatrix}\,, \quad
        d^1_{\mathcal{E}_n^*} = 
      \begin{pmatrix} 
        \divis & n 
      \end{pmatrix}\,, \quad
    d^0_{\Cone(\hat{u})^*} = 
      \begin{pmatrix} \divis \\ 
        1-\sigma' \\ 
        -n 
      \end{pmatrix}\,, 
  \]
  \[
    d^1_{\Cone(\hat{u})^*}= 
    \begin{pmatrix} 
      \sigma'-1 & \divis & 0 \\ 
      n & 0 & \divis \\ 
      0 & - N_{Y|X} & 0 \\ 
      0 & -n & \sigma'-1 
    \end{pmatrix}\,, \quad 
    d^2_{\Cone(\hat{u})^*} = 	\begin{pmatrix} 
      N_{Y|X} & 0 & \divis & 0 \\ 
      n & 1-\sigma' & 0 & \divis \\ 
      0 & 0 & n & -N_{Y|X} 
    \end{pmatrix}\,, 
  \]  
  \[
    d^3_{\Cone(\hat{u})^*} = 
    \begin{pmatrix} 
      -n & N_{Y|X} & \divis 
    \end{pmatrix}\,, \quad 
    d^1_{\mathcal{K}[1]^*} = 
    \begin{pmatrix} 
      \divis \\ 
      - N_{Y|X} \\ 
      - n 
    \end{pmatrix}
  \]  
  \[
    d^2_{\mathcal{K}[1]^*} = 
    \begin{pmatrix} 
      N_{Y|X} & \divis & 0 \\ 
      n & 0 & \divis \\ 
      0 & n & -N_{Y|X} 
    \end{pmatrix}\,, \quad
    d^3_{\mathcal{K}[1]^*} =  d^3_{\Cone(\hat{u})^*}
  \] 
and the differentials for $\mathcal{E}_n[1]^*$ are given by $d^i_{\mathcal{E}_n[1]^*} = d^{i-1}_{\mathcal{E}_n^*}$. The maps are given as follows: $q(\hat{u})^*_0(a)= i(a), $ 
  \[
    q(\hat{u})^*_1(a,b) = (i(b),0,i(a))\,,
  \] 
  \[
    q(\hat{u})^*_2(a) = (0,i(a),0,0)\,,
  \] 
  \[
    (\pr_2^*)_1(a) =(0,a,0)\,,
  \] 
  \[
    (\pr_2^*)_2(a,b,c) = (a,0,b,c)\,,
  \] 
$(\pr_2^*)_3$ and $(\pr_2^*)_4$ are equal to the identity map, while $f^*_1 = i$, 
  \[
    f^*_2(a,b) = (i(b),a,i(a))\,,
  \] 
  \[
    f^*_3(a) = (a,i(a),0)\,.
  \]
We now proceed with the proof of Lemma \ref{lemma:isoh}.

\begin{proof}[{Proof of Lemma \ref{lemma:isoh}}]
We start by running the hypercohomology spectral sequence on $\HOM(\Cone(\hat{u}),\mathcal{C}).$ Then we have \[E_1^{p,q}=H^q(\tilde{X}, \HOM^p(C(\hat{u}),\C))\] and the complex $\HOM(C(\hat{u}),\C)$ has the form
  \[
    \begin{tikzcd}
      \pi_*\pi^*j_*\GG_{m,K}\to 
      \begin{matrix}
        \pi_*\pi^*\DIV X \\
        \oplus \\
        \pi_*\pi^*j_*\GG_{m,K} \\
        \oplus \\
        \pi_*\pi^*j_*\GG_{m,K}
      \end{matrix}\to 
      \begin{matrix}
        \pi_*\pi^*\DIV X \\
        \oplus \\
        \pi_*\pi^*\DIV X \\
        \oplus \\
        j_*\GG_{m,K} \\
        \oplus \\
        \pi_*\pi^*j_*\GG_{m,K}
      \end{matrix}\to
      \begin{matrix}
        \DIV X \\
        \oplus \\
        \pi_*\pi^*\DIV X \\
        \oplus \\
        j_*\GG_{m,K}
      \end{matrix}\to
      \DIV X\,.
    \end{tikzcd}
  \]
Since $\pi$ is finite \'etale, the pushforward $\pi_*$ is exact and, using the Leray spectral sequence, we see that $H^n(X, \pi_*F)\cong H^n(Y, F)$ for all \'etale sheaves $F$ on $Y$ and all $n\geq 0$. Furthermore, since $\pi$ is \'etale, we have an isomorphism $\pi^*\DIV X\cong \DIV Y$ and $\pi^*j_*\GG_{m,K}$ is isomorphic a direct sum of sheaves $(j_L)_*\GG_{m,L}$, one for each component of $Y$, where $j_L$ is the inclusion of the generic point on $\Spec \OO_L$. It follows that for every $p$, the sheaf $\HOM^p(C(\hat{u}),\C)$ is a direct sum of sheaves which by Proposition \ref{prop:Bienencoh} only has cohomology in degree 0 and 2. Hence the $E_1$-page can be visualized as
  
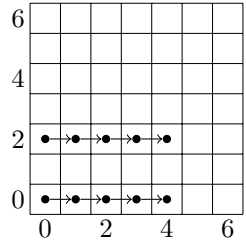
\[
    \begin{sseq}{0...6}{0...6}
      \ssdropbull
      \ssarrow{1}{0}
      \ssdropbull
      \ssarrow{1}{0}
      \ssdropbull
      \ssarrow{1}{0}
      \ssdropbull
      \ssarrow{1}{0}
      \ssdropbull
      \ssmoveto{0}{2}
      \ssdropbull
      \ssarrow{1}{0}
      \ssdropbull
      \ssarrow{1}{0}
      \ssdropbull
      \ssarrow{1}{0}
      \ssdropbull
      \ssarrow{1}{0}
      \ssdropbull
    \end{sseq}
  \]
Recall that $\tilde{Y} \cong \Ind^{C_n}_{C_d}(\tilde{Z})$ for $Z = \Spec \OO_L$ the ring of integers of an unramified extension $L$ of $K$ and recall that $\Br_0 L$ denotes the classes of $\Br L$ which has zero local invariants at all real places. The differential 
  \[
    d_1^{0,2} \colon E_1^{0,2} = (\Br_0 L)^{n/d} \rightarrow (\bigoplus_{z \in Z} \Br L_z)^{n/d} \oplus (\Br_0 L)^{n/d} \oplus (\Br_0 L)^{n/d}
  \] 
is injective, since the map $(\Br_0 L)^{n/d} \rightarrow (\bigoplus_{z \in Z} \Br L_z)^{n/d}$ coincides with the map $(\inv)^{n/d}$, where $\inv$ is the restriction to $\Br_0 L$ of the invariant map $\Br L \rightarrow \bigoplus_{z} \Br L_z$, where $z$ ranges over all places of $L$. This restriction is injective since the subgroup $\Br_0 L$ has no invariants at the complex points. On the $E_2$-page, we see that no differential can hit $E_2^{p,0}$ for $p=0,1,2$. This shows that 
  \[
    E_2^{p,0}=E_\infty^{p,0} = H^p(\RHom(\Cone(q(\hat{u})),\mathcal{C})) 
  \] 
for $p=0,1,2$. But 
  \[
    E_2^{p,0}=H^p(\Hom(C(\hat{u}),\C))
  \] 
and hence $H^i(t)$ is an isomorphism for $i=0,1,2$. If one uses the hypercohomology spectral sequence on $\HOM(\mathcal{E}_n,\mathcal{C})$, one sees that $H^i(s)$ is an isomorphism for $i=0,1,2$ as well. The claim that $q(\hat{u})^*$ induces an isomorphism now follows. Indeed, we know that $\RHom(q(\hat{u}),\mathcal{C})$ is an isomorphism, so that $H^i(\RHom(q(\hat{u}),\mathcal{C}) \circ s) = H^i(t \circ q(\hat{u})^*)$ is an isomorphism for $i=0,1,2$. Since $H^i(t)$ is an isomorphism in these degrees, the statement follows.
\end{proof}

\begin{cor} \label{cor:concomp}
The map $c_x^\sim \colon \Ext^i_{\tilde{X}}(\ZZ/n\ZZ_{\tilde{X}},\phi_* \mathbb{G}_{m,X}) \rightarrow \Ext^{i+1}_{\tilde{X}}(\ZZ/n\ZZ_{\tilde{X}}, \phi_* \mathbb{G}_{m,X})$ for $i=0,1$ is isomorphic to $H^i(q(\hat{u})^*)^{-1} \circ H^i (\pr_2^*) \circ H^i(\hat{f}^*)$.
\end{cor}

We will now utilize this corollary to compute the map $c_x^\sim.$ We will use the maps and the notation of the diagram on page \pageref{diag} freely in the proofs that follow. If $x \in H^1(\tilde{X},\ZZ/n\ZZ)$ is represented by a $\ZZ/n\ZZ$-torsor of the form $\Ind_{\ZZ/d\ZZ}^{\ZZ/n\ZZ}((Y,Y_\infty)) \rightarrow X$, for $Y = \Spec \OO_L$ the ring of integers of an unramified cyclic extension $L/K$ of degree $d$, we will say that we identify $x$ with the cyclic extension $L/K$ of degree $d|n$ together with a choice of generator $\sigma\in \Gal(L/K)$.

\begin{lem} \label{lem:cap1}
Let $x\in H^1(\tilde{X},\mathbb{Z}/n)$ and identify $x$ with a pair $(L,\sigma)$, where $L/K$ is a cyclic extension of degree $d|n$ that is unramified at all places, including the infinite ones, and $\sigma\in \Gal(L/K)$ is a generator. Then the morphism 
  \[
    c_x^\sim \colon \Ext^0_{\tilde{X}}(\mathbb{Z}/n\ZZ_{\tilde{X}},\phi_* \mathbb{G}_{m,X})\to\Ext^1_{\tilde{X}}(\mathbb{Z}/n\ZZ_{\tilde{X}},\phi_*\mathbb{G}_{m,X})
  \] 
sends $\xi\in \mu_n(K)\cong  \Ext^0_{\tilde{X}}(\mathbb{Z}/n\ZZ_{\tilde{X}},\phi_* \mathbb{G}_{m,X})$ to 
  \[
    (a,I)\in H^1(\Hom(\mathcal{E}_n,\mathcal{C})) \cong \Ext^1_{\tilde{X}}(\mathbb{Z}/n\ZZ_{\tilde{X}},\phi_* \mathbb{G}_{m,X})
  \] 
where $a \in K^\times$ and $I \in \Div K$ are elements such that $a=b^{-n}$ and $I\OO_L=\divis(b)$ respectively, where $b \in L^\times$ is an element satisfying $\xi^{n/d}=\sigma(b)/b$ in $L^\times$.
\end{lem}

\begin{proof}
We use Corollary \ref{cor:concomp}. We see that $(\pr_2^*)_1 \circ f^*_1$ takes the element $\xi \in \mu_n(K)$ to $(0,i(\xi),1)$. We want to reduce this element, modulo the image of $d^0_{\Cone(\hat{u})^*}$ so that it is in the image of $q(\hat{u})^*_1$. Note that the norm of $\xi^{n/d} \in L^\times$ is $\xi^n=1$, thus Hilbert's theorem 90 gives us an element $b \in L^\times$ such that 
  \begin{equation}\label{h90} 
    \xi^{n/d} = \sigma(b)/b \,.
  \end{equation} 
Let $\underline{b} = (b,\xi^{-1}b, \xi^{-2}b, \ldots, \xi^{-n/d+1}b) \in (L^\times)^{n/d}$. We reduce $(0,i(\xi),1)$ by the element $d^0_{\Cone(\hat{u})}(\underline{b})= (\divis(\underline{b}),\underline{b}/\sigma'(\underline{b}),\underline{b}^{-n})$ to get the element $(\divis(\underline{b}),1,\underline{b}^{-n})$. Since $\xi$ is a unit, we have 
  \[
    \divis(\underline{b}) = (\divis(b), \divis(b), \ldots, \divis(b)) \in (\Div L)^{n/d}\,,
  \] 
and since $\xi^n=1$, we have $\underline{b}^{-n} = (b^{-n}, \ldots, b^{-n}) \in (L^\times)^{n/d}$. To show that this is in the image of the $q(\hat{u})_1^*$, we note that $\divis(b)$ is invariant under the Galois action. Indeed, taking $\divis$ of the equality (\ref{h90}), we see that $\divis(b) = \sigma(\divis(b))$. This implies that there is a fractional ideal $I$ of $\Div K$ such that $i(I) = \divis(\underline{b})$. Similarly, we see that $b^{-n}$ is invariant under the Galois action, so that there is some element $a \in K^\times$ such that $i(a) = \underline{b}^{-n}$. It is then clear that $q(\hat{u})^*(a,I) = (\divis(\underline{b}),1,\underline{b}^{-n})$, so that our lemma follows.
\end{proof}

\begin{cor} \label{cor:caproots}
Suppose that in Lemma \ref{lem:cap1} the field $K$ contains a primitive $n$th root of unity. Identify $L$ with a Kummer extension $L=K(v^{1/n})$ with $v \in K^\times$ such that $\divis(v) = n \mathfrak{a}$ for some $\mathfrak{a} \in \Div(K)$. Choose a primitive root of unity $\xi$ such that $\xi^{n/d} = \sigma(v^{1/n})/v^{1/n}$. Then  
  \[
    c_x^\sim \colon \Ext^0_{\tilde{X}}(\mathbb{Z}/n\ZZ_{\tilde{X}},\phi_* \mathbb{G}_{m,X})\to\Ext^1_{\tilde{X}}(\mathbb{Z}/n\ZZ_{\tilde{X}},\phi_*\mathbb{G}_{m,X})
  \] 
sends the element $\xi \in \mu_n(K)\cong  \Ext^0_{\tilde{X}}(\mathbb{Z}/n\ZZ_{\tilde{X}},\phi_* \mathbb{G}_{m,X})$ to 
  \[
    (v^{-1},\mathfrak{a})  \in H^1(\Hom(\mathcal{E}_n,\mathcal{C})) \cong \Ext^1_{\tilde{X}}(\mathbb{Z}/n\ZZ_{\tilde{X}},\phi_* \mathbb{G}_{m,X})\,.
  \]
\end{cor}

\begin{proof}
In the statement of \ref{lem:cap1} we can in this situation choose $b$ to be $v^{1/n}.$ Since clearly $v^{-1}$ lies under $v^{-1/n}$ and $\mathfrak{a}$ lies under $\divis(v^{1/n})$, our Corollary follows.
\end{proof}
Note that since $\xi$ is a primitive $n$th root of unity, Corollary \ref{cor:caproots} determines the map $c_x^\sim$ in full. We now determine the map 
  \[
    c_x^\sim \colon  \Ext^1_{\tilde{X}}(\ZZ/n\ZZ_{\tilde{X}},\phi_* \mathbb{G}_{m,X}) \rightarrow \Ext^{2}_{\tilde{X}}(\ZZ/n\ZZ_{\tilde{X}}, \phi_* \mathbb{G}_{m,X})\,.
  \]

\begin{lem} \label{lem:cap2}
Let $x\in H^1(\tilde{X},\mathbb{Z}/n)$ and identify $x$ with a pair $(L,\sigma)$, where $L/K$ is a cyclic extension of degree $d|n$ that is unramified at all places, including the infinite ones, and $\sigma$ is a generator of $\Gal(L/K)$. Then the morphism 
  \[ 
    c_x^\sim \colon  \Ext^1_{\tilde{X}}(\ZZ/n\ZZ_{\tilde{X}},\phi_* \mathbb{G}_{m,X}) \rightarrow \Ext^{2}_{\tilde{X}}(\ZZ/n\ZZ_{\tilde{X}}, \phi_* \mathbb{G}_{m,X})
  \] 
sends the element $(b, \mathfrak{b}) \in H^1(\Hom(\mathcal{E}_n,\mathcal{C})) \cong \Ext^1_{\tilde{X}}(\mathbb{Z}/n\ZZ_{\tilde{X}},\phi_*\mathbb{G}_{m,X})$ to 
  \[
    \frac{n}{d}N_{L|K}(I)+ \frac{n^2}{2d} \mathfrak{b} \in \Cl K/n \Cl K \cong \Ext^2_{\tilde{X}}(\mathbb{Z}/n\ZZ_{\tilde{X}},\phi_*\mathbb{G}_{m,X})\,.
  \] 
In this formula, $I$ is any fractional ideal of $L$ satisfying the equality 
  \[
    \mathfrak{b}^{n/d}\mathcal{O}_L=I-\sigma(I)+\divis(t)\,,
  \] 
where $t \in L^\times$ satisfies $N_{L|K}(t) = b^{-1}$.
\end{lem}

\begin{proof}
Once again, we use Corollary \ref{cor:concomp}. We see that $(\pr_2^*) \circ f^*_2$ takes the element 
  \[
    (b,\mathfrak{b})\in K^\times\oplus \Div K
  \] 
to 
  \[
    (i(\mathfrak{b}),0,b,i(b)) \in (\Div L)^{n/d} \oplus (\Div L)^{n/d} \oplus K^\times \oplus (L^\times)^{n/d}\,.
  \] 
We want to reduce this element modulo the image of $d^1_{\Cone(\hat{u})}$ to get an element of the form $(0,i(J),1,1)$ for some ideal $J \in \Div K$, since $\im(q(\hat{u})_2) = (0,i(\Div(K)),1,1)$. Let $\sigma$ be the generator of $\Gal(L/K)$ corresponding to $x$. Since $\mathfrak{b}^{n/d}\mathcal{O}_L$ is in the kernel of the map $N_{L|K} \colon \Cl L \rightarrow \Cl K$, Furtwängler's theorem \cite[IV, Theorem 1]{LemmermeyerPrincipal} gives us a fractional ideal $\mathfrak{b}' \in \Div L$ and an element $a \in L^\times$ such that 
  \begin{equation}\label{eq:furt}
    \mathfrak{b}^{n/d}\mathcal{O}_L=\mathfrak{b}'-\sigma(\mathfrak{b}')+\divis(a)\,.
  \end{equation}  
Note that $\divis(N_{L|K}(a)) = n\mathfrak{b} = \divis(b^{-1})$, so that $N_{L|K}(a) = b^{-1}u$ for some unit $u \in K^\times$. Since units are always norms in unramified extensions of local fields, Hasse's norm theorem \cite{HasseNorm} implies that there is a $v \in L^\times$ such that $N_{L|K}(v) =u^{-1}$. Now, since $N_{L|K}(\divis(v))$ is the unit ideal, Hilbert's theorem 90 for ideals (see e.g. \cite[Proposition 2.1.1]{Bembom}) implies that there is an ideal $J \in \Div L$ such that $\divis(v) = J-\sigma(J)$. Set $I = \mathfrak{b}'-J$ and $t=av$. Then we see that 
  \[
    \mathfrak{b}^{n/d}\mathcal{O}_L = (\mathfrak{b}'-J)-\sigma(\mathfrak{b'}-J)+\divis(av) = I-\sigma(I)+\divis(t)\,.
  \]  
Set 
  \[
    \underline{I} = (I,\mathfrak{b}+I,2\mathfrak{b}+I, \ldots, (\dfrac{n}{d}-1) \mathfrak{b}+I) \in (\Div L)^{n/d}
  \] 
and $\underline{t} = (t,1, \ldots , 1)$. Denote by $\divis(\underline{t})$ and $\underline{t}^n$ respectively the element obtained by applying the operators $\divis$ and $(-)^n$ component-wise and recall the definition of $\sigma'-1$ after Lemma \ref{lemma:isoh}. Then 
  \[
    (\sigma'-1)\underline{I} = (\frac{n}{d}\mathfrak{b}+(\sigma-1)I-\mathfrak{b}, -\mathfrak{b}, \dots, -\mathfrak{b})\,.  
  \]
If we reduce $(i(\mathfrak{b}),0,b,i(b))$ modulo the element 
  \[
    d^1_{\Cone(q(\hat{u}))^*}(\underline{I},\underline{t}^{-1},1) = ((\sigma'-1)(\underline{I})-\divis(\underline{t}),n\underline{I},b^{-1},\underline{t}^{n})
  \] 
we get the element 
  \[
    (0,n\underline{I},1,\underline{t}^ni(b))\,.
  \] 
Now, $t^nb^{n/d}$ is in the kernel of $N_{L|K}\colon L^\times \rightarrow K^\times$, so by Hilbert's theorem 90 there is some element $w \in L^\times$ such that $w/\sigma(w) = t^{n}b^{n/d}$. We claim that we can choose the element $w$ as 
  \[
    t^n \sigma(t^{n-\frac{n}{d}}) \sigma^2(t^{n-\frac{2n}{d}}) \cdots \sigma^{d-1}(t^{n-\frac{n(d-1)}{d}})\,.
  \] 
Indeed, we have
  \[
    \begin{split} 
      w/\sigma(w) & = \dfrac{t^n \sigma(t^{n-\frac{n}{d}}) \sigma^2(t^{n-\frac{2n}{d}}) \cdots \sigma^{d-1}(t^{n-\frac{n(d-1)}{d}})}{\sigma(t^n \sigma(t^{n-\frac{n}{d}}) \sigma^2(t^{n-\frac{2n}{d}}) \cdots \sigma^{d-1}(t^{n-\frac{n(d-1)}{d}}))} \\ 
      & = \dfrac{t^n \sigma(t^{n-\frac{n}{d}}) \sigma^2(t^{n-\frac{2n}{d}}) \cdots \sigma^{d-1}(t^{n-\frac{n(d-1)}{d}})}{\sigma(t^n) \sigma^2(t^{n-\frac{n}{d}}) \sigma^3(t^{n-\frac{2n}{d}}) \cdots \sigma^{d-1}(t^{n-\frac{n(d-2)}{d}}) \cdot t^{n-\frac{n(d-1)}{d}}} \\ 
      & = \dfrac{t^n}{N_{L|K}(t)^{n/d}} \\ 
      & = t^nb^{n/d}\,.
    \end{split}
  \] 
Set $\underline{w} = (w,bw,b^2w, \cdots ,b^{n/d-1}w).$
We now reduce the element $(0,n\underline{I},1,\underline{t}^ni(b))$ modulo 
  \[
    d^1_{\Cone(q(\hat{u}))^*}(0,0,\underline{w}) = (0,\divis(\underline{w}),0,\underline{t}^{-n}i(b)^{-1})
  \] 
and get the element $(0,n\underline{I}+\divis(\underline{w}),1,1)$. Note that 
  \[
    n \underline{I} + \divis(\underline{w}) = (nI + \divis(w),nI + \divis(w) , \ldots, nI+\divis(w))\,.
  \] 
The equality $\divis(t) = \mathfrak{b}^{n/d}-I+\sigma(I)$ shows that
  \[
    \begin{split}
      \divis(w) & = n((\mathfrak{b}^{n/d})-I+\sigma(I)) +(n-\frac{n}{d}) (\mathfrak{b}^{n/d}-\sigma(I)+\sigma^2(I)) + \cdots +\frac{n}{d}(\mathfrak{b}^{n/d}-\sigma^{d-1}(I)+I) \\ 
      & = \frac{n(d+1)}{2}\mathfrak{b}^{n/d} -(n-\frac{n}{d})I +\frac{n}{d}( \sigma(I)+ \sigma^2(I) + \cdots + \sigma^{d-1}(I) )\,,
    \end{split}
  \]  
and hence 
  \[
    (0,n\underline{I}+\divis(w),0,0) = (0,i(\frac{n(d+1)}{2}\mathfrak{b}^{n/d}+N_{L|K}(I)^{n/d}),0,0)\,.
  \]   
The proposition now follows since 
  \[
    \frac{n(d+1)}{2}\mathfrak{b}^{n/d} + N_{L|K}(I)^{n/d} = \frac{n^2}{2d}\mathfrak{b}+\dfrac{n}{d}N_{L|K}(I)
  \] 
in $\Cl K /n \Cl K.$
\end{proof}

\begin{rk}\label{rk:sharifi}
Let $K$ be a number field containing the group of $n$th roots of unity and let $S$ be a finite set of places containing all the places lying over $n$ as well as the real archimedean places. Let us denote by $G_{K,S}$ the Galois group of the maximal extension of $K$ unramified outside of $S$ and denote by $K_S$ the maximal extension of $K$ unramified outside of $S$. Let us note that if we set $D_K = (K_S^\times)^n \cap K^\times$, we have the isomorphism $H^1(G_{K,S},\mu_n) \cong D_K/(K^\times)^n$.  In \cite{Sharificup}, McCallum--Sharifi compute the cup product
  \[
    (-,-)_S\colon D_K \otimes D_K \cong H^1(G_{K,S},\mu_n) \otimes H^1(G_{K,S},\mu_n) \rightarrow H^2(G_{K,S},\mu_n^{\otimes 2})\,.
  \] 
If $a,b \in D_K$ are elements such that all local Hilbert pairings $(a,b)_v$ with $v \in S$ are trivial, one can identify the element $(a,b)_S \in H^2(G_{K,S},\mu_n^{\otimes 2})$ with an element of $(\Cl K_S/n \Cl K_S) \otimes \mu_n $. Under the assumptions that these local pairings vanish, \cite[Theorem 2.4]{Sharificup} gives a formula for $(a,b)_S$ which is similar to the formula in Lemma \ref{lem:cap2}. We now describe the formula for $(a,b)_S$ obtained by McCallum--Sharifi. Note that there is some $\alpha \in K^\times_S$ such that $\alpha^n = a.$ Let $L = K(\alpha)$, and $d = [L:K]$. Let $\mathfrak{b} \in \Div K_S$ be such that $\divis(b)\mathcal{O}_{K,S} = \mathfrak{b}^n$. Write $b = N_{L|K} \gamma$ for some $\gamma$ and $\divis(\gamma) \OO_{L,S} = I-\sigma(I)+b^{n/d}$ for some $I \in \Div L_S$ and $\sigma \in \Gal(L/K)$. Let $\xi \in \mu_n$ be such that $\sigma (\alpha) = \xi \alpha$. Then 
  \[
    (a,b)_S = (N_{L|K}(I) + \frac{n}{2}\mathfrak{b}) \otimes \xi \in \Cl K_S \otimes \mu_n\,.
  \]
\end{rk}

\begin{rk}\label{rk:units}
If $x \in H^1(\tilde{X},\ZZ/n\ZZ)$ is represented by an extension $L\supset K$ whose ring of integers contains a unit $t$ of norm $\xi$ for every unit $\xi\in \mathcal{O}_K^\times$, then the image of $(b,\mathfrak{b})$ only depends on $\mathfrak{b}$ and not on $b$. Indeed, $\mathfrak{b}$ determines $b$ up to a unit and if $\xi$ is a unit in $\mathcal{O}_K$, then for the element $(\xi^{-1},(1))$ we may choose $I=(1)$ and $t$ a unit of norm $\xi$ in Lemma \ref{lem:cap2}. For example, this holds if $\mu(K)=\{1,-1\}$ and $d$ is odd.
\end{rk}

\begin{rk}
Let $n$ be an odd positive integer and let $K$ be a number field. Then 2 is invertible in $\ZZ/n\ZZ$ and hence the formula above becomes \[2^{-1}\frac{n}{d}n\mathfrak{b}+\frac{n}{d}N_{L|K}(I)=\frac{n}{d}N_{L|K}(I)\] since $n\mathfrak{b}=0$ mod $n$.
\end{rk}

\begin{rk}
Let $K$ be a number field with class group $\ZZ/p^n\ZZ$ for some prime $p$ and some positive integer $n\geq 1$. Let $x \in H^1(\tilde{X},\ZZ/p^n\ZZ)$ be represented by a cyclic extension $L/K$ of degree $d|p^n$, unramified at all places (including the infinite ones), together with a choice of generator $\sigma \in \Gal(L/K)$. Let $\mathfrak{b}$ be an ideal in $K$. We claim that $\mathfrak{b}^{p^n/d}\mathcal{O}_L$ is principal. 
Indeed, $\mathfrak{b}^{p^n/d}\mathcal{O}_L$ is $d$-torsion and, by class field theory, $L$ is the unique subextension of degree $d$ of the Hilbert class field of $K$ and every fractional ideal representing a $d$-torsion class in $\Cl(K)$ becomes principal once extended to $L$. 
Hence, if we pick $t$ to be a generator of $\mathfrak{b}^{p^n/d}\mathcal{O}_L$, we get that \[c_x^\sim(N_{L|K}(t)^{-1},\mathfrak{b})=\frac{(p^n)^2}{2d}\mathfrak{b}\,.\]
In particular, if $p$ is an odd prime and $\mu(K)=\{1,-1\}$, then Remark \ref{rk:units} shows that \[c_x^\sim(b,\mathfrak{b})=0\,.\]
\end{rk}

We are now finally in a position to compute the cup product map 
  \[
    c_x\colon H^i(\tilde{X},\ZZ/n\ZZ) \rightarrow H^{i+1}(\tilde{X},\ZZ/n\ZZ)\,.
  \] 
From now on, given $z \in \Ext^{3-i}_{\tilde{X}}(\ZZ/n\ZZ_{\tilde{X}},\phi_* \mathbb{G}_{m,X})$ and $y \in \Ext^{3-i}_{\tilde{X}}(\ZZ/n\ZZ_{\tilde{X}},\phi_* \mathbb{G}_{m,X})^\sim$, where $^\sim$ is the Pontryagin dual, we let $\langle y,z\rangle$ be the evaluation of $y$ on $z$.

\begin{prop}\label{prop:cup1}
Let $X = \Spec \OO_K$ be the ring of integers of a number field $K$ and identify $H^2(\tilde{X},\ZZ/n\ZZ)$ and $H^3(\tilde{X},\ZZ/n\ZZ)$ with $\Ext^1_{\tilde{X}}(\ZZ/n\ZZ_{\tilde{X}},\phi_* \mathbb{G}_{m,X})^\sim$ and $\mu_n(K)^\sim$ respectively, where $^\sim$ denotes the Pontryagin dual. Let $x \in H^1(\tilde{X},\ZZ/n\ZZ)$ be represented by a pair $(L,\sigma)$, where $L/K$ is a cyclic extension of degree $d|n$, unramified at all places (including the infinite ones), and $\sigma \in \Gal(L/K)$ is a generator. Let $\xi \in \mu_n(K)$ and choose $b \in L^\times$ such that $\xi^{n/d} = \sigma(b)/b$, and let $a \in K^\times$ and $\mathfrak{a} \in \Div K$ be such that $\mathfrak{a}\OO_L = \divis(b)$ and $a = b^{-n}$ in $L^\times$. If $y \in H^2(\tilde{X},\ZZ/n\ZZ)$, then 
  \[
    \langle x \cup y,\xi\rangle =  \langle y,(a,\mathfrak{a})\rangle\,.
  \]
\end{prop}

\begin{proof}
This follows at once from Lemma \ref{lem:cap1}.
\end{proof}

\begin{cor}\label{cor:cup1nonzero}
Let $X = \Spec \OO_K$ be the ring of integers of a number field $K$ containing all $n$th roots of unity, and identify $H^2(\tilde{X},\ZZ/n\ZZ)$ and $H^3(\tilde{X},\ZZ/n\ZZ)$ with $\Ext^1_{\tilde{X}}(\ZZ/n\ZZ_{\tilde{X}},\phi_* \mathbb{G}_{m,X})^\sim$  and $\mu_n(K)^\sim$ respectively, where $^\sim$ denotes the Pontryagin dual. Let $x \in H^1(\tilde{X},\ZZ/n\ZZ)$ be represented by a pair $(L,\sigma)$, where $L=K(v^{1/n})$ Kummer extension of degree $d = [L:K]$, with $v \in K^\times$ such that $\divis(v) = n \mathfrak{a}$ for some $\mathfrak{a}$ in $\Div(K)$, and $\sigma\in \Gal(L/K)$ is a generator. If $y \in H^2(\tilde{X},\ZZ/n\ZZ)$, then $x \cup y \neq 0$ if and only if $\langle y,(v^{-1},\mathfrak{a})\rangle \neq 0$.
\end{cor}

\begin{proof}
This is a direct consequence of Corollary \ref{cor:caproots}. To see this, let $\xi$ be a primitive $n$th root of unity such that $\sigma(v^{1/n})/v^{1/n} = \xi^{n/d}$. We then see that $c_x(y)(\xi) = \langle y, (v^{-1},\mathfrak{a}) \rangle$. Since $\xi$ generates $\mu_n(K)$, the element $c_x(y) \in \mu_n(K)^\sim$ is zero if and only if $\langle c_x(y),\xi\rangle=0$, so we have our Corollary.
\end{proof}

\begin{prop}\label{prop:cup2}
Let $X = \Spec \OO_K$ be the ring of integers of a number field $K$ and identify $H^2(\tilde{X},\ZZ/n\ZZ)$ with $\Ext^1_{\tilde{X}}(\ZZ/n\ZZ_{\tilde{X}},\phi_* \mathbb{G}_{m,X})^\sim$, where $^\sim$ denotes the Pontryagin dual. Let $x \in H^1(\tilde{X},\ZZ/n\ZZ)$ be represented by a pair $(L,\sigma)$, where $L/K$ is a cyclic extension of degree $d|n$, unramified at all places (including the infinite ones), and $\sigma \in \Gal(L/K)$ is a generator. For an element $y \in H^1(\tilde{X},\ZZ/n\ZZ) \cong (\Cl K/n \Cl K)^\sim$ represented by an unramified cyclic extension $M/K$, we have that 
  \[
    x \cup y \in \Ext^1_{\tilde{X}}(\ZZ/n\ZZ_{\tilde{X}},\phi_*\mathbb{G}_{m,X})^\sim
  \] 
satisfies the formula 
  \[
    \langle x \cup y,(a,\mathfrak{b})\rangle =  \langle y,N_{L|K}(I)^{n/d}+\frac{n^2}{2d} \mathfrak{b}\rangle
  \] 
where $(a,\mathfrak{b}) \in \Ext^1_{\tilde{X}}(\ZZ/n\ZZ_{\tilde{X}},\phi_*\mathbb{G}_{m,X})$ and $I \in \Div L$ is any fractional ideal such that $\mathfrak{b}^{n/d} \mathcal{O}_L =I-\sigma(I) + \divis(t)$ for some $t \in L^\times$ such that $N_{L|K}(t) =a^{-1}$. In particular, $\langle x \cup y,(a,\mathfrak{b})\rangle = 0$ if and only if $\frac{n^2}{2d} \mathfrak{b}+N_{L|K}(I)^{n/d}$ is in the image of $N_{M|K}$.
\end{prop}

\begin{proof}
By Artin reciprocity, $y \in H^1(\tilde{X},\ZZ/n\ZZ)$ corresponds to a map 
  \[
    \Cl K /n \Cl K \rightarrow \ZZ/n\ZZ
  \] 
with kernel $N_{M|K}(\Cl M)$. The formula for $x \cup y$ is given by Lemma \ref{lem:cap2}. Indeed, 
  \[
    \langle x \cup y, (a, \mathfrak{b})\rangle = \langle y,c_x^\sim(a, \mathfrak{b})\rangle = \langle y, \frac{n^2}{2d}\mathfrak{b}+N_{L|K}(I)^{n/d}\rangle
  \] 
where $I$ is as in the proposition. The fact that $\langle x \cup y, (a,\mathfrak{b})\rangle=0$ if and only if $\frac{n^2}{2d}\mathfrak{b}+N_{L|K}(I)^{n/d}$ is in the image of $N_{M|K}$ readily follows from this formula and the observation that the kernel of $y\colon \Cl K /n \Cl K \rightarrow \ZZ/n\ZZ$ is $N_{M|K}(\Cl M)$.
\end{proof}

\begin{ex}
The imaginary quadratic fields $\QQ(\sqrt{-231})$ and $\QQ(\sqrt{-255})$ has isomorphic cohomology groups in all degrees. The class groups modulo 2 both have $\FF_2$-rank $2$ and hence in both cases $H^1(X, \ZZ/2\ZZ)$ has rank 2 and $H^2(X, \ZZ/2\ZZ)$ has rank 3 since both fields has a non-trivial 2nd root of unity. Choosing a basis $x_1, x_2$ for $H^1(X, \ZZ/2\ZZ)$ and a basis $\varphi_1, \varphi_2, \varphi_3$ we may compute the matrix 
  \[
    M_K:=\langle x_i\cup x_j, \varphi_k\rangle_{1\leq i\leq j\leq 2, 1\leq k\leq 3}  
  \]
which determines the cup product $H^1(X, \ZZ/2\ZZ)\otimes H^1(X, \ZZ/2\ZZ)\to H^2(X, \ZZ/2\ZZ)$. This was done using the computer program https://github.com/ericahlqvist/cup-products, written in C using the library PARI \cite{PARI2}. The first example $K=\QQ(\sqrt{-231})$ yields a matrix $M_K$ of rank 3 whereas the second example $K=\QQ(\sqrt{-255})$ yields a matrix $M_K$ of rank 2 and hence the cohomology rings cannot be isomorphic. 
\end{ex}


\section{A non-vanishing criterion for Kim's invariant} \label{sec:kim}
In this section we apply the results in Section \ref{sec:cup} to give a criterion for when an interesting invariant associated to number fields defined by Minhyong Kim in \cite{KimArithmetic} vanishes. Following \cite{BleherUnramified}, we call this invariant Kim's invariant. Let us briefly recall how this invariant is defined. Let $K$ be a totally imaginary number field and assume that $K$ contains a primitive $n$th root of unity and set $X= \Spec \OO_K$. By choosing a root of unity $\xi$, fix an isomorphism $\xi \colon \ZZ/n\ZZ \cong \mmu_{n,X}$. Let $\inv \colon H^3(X,\mmu_n) \rightarrow \ZZ/n\ZZ$ be the usual invariant isomorphism obtained via global class field theory. We then have an isomorphism 
  \[
    H^3(X,\ZZ/n\ZZ) \xrightarrow{\xi_*} H^3(X,\mmu_n) \cong \ZZ/n\ZZ
  \] 
which we also will denote by $\inv$. Let now $G$ be a finite group and assume that we have a morphism $f \colon \pi_1(X,x) \rightarrow G$. Fix a class $c \in H^3(G,\ZZ/n\ZZ)$. We then get a class $f^*(c) \in H^3(\pi_1(X,x),\ZZ/n\ZZ)$. We have a map of sites $k\colon\Xet \rightarrow B\pi_1(X,x)$ and if we let 
  \[
    f^*_X = k^* \circ f^*\colon H^3(G,\ZZ/n\ZZ) \rightarrow H^3(X,\ZZ/n\ZZ)\,,
  \] 
we get an element $f^*_X(c) \in H^3(X,\ZZ/n\ZZ)$. Then Kim's invariant associated to $f$ and $c$ is the element 
  \[
    \inv(f^*_X(c)) \in \ZZ/n\ZZ\,. 
  \]
An interesting problem is to compute this invariant for the case when $G= C_n$, and this is what we now will do. We let $c_1 \in H^1(C_n,\ZZ/n\ZZ)$ be the class corresponding to the identity element and let $c_2 = \beta_n(c_1) \in H^2(C_n,\ZZ/n\ZZ)$ be the image of $c_1$ under the Bockstein homomorphism $\beta_n \colon H^1(C_n,\ZZ/n\ZZ) \rightarrow H^2(C_n,\ZZ/n\ZZ)$ arising from the short exact sequence
  \begin{equation}\label{eq:bock} 
    0 \rightarrow \ZZ/n\ZZ \xrightarrow{n} \ZZ/n^2\ZZ \rightarrow \ZZ/n\ZZ \rightarrow 0\,. 
  \end{equation} 
Then it is well-known (see \cite[V,Ex.4.3]{BrownCohomology}) that $H^3(C_n,\ZZ/n\ZZ) \cong \ZZ/n\ZZ$ and that the class $c = c_1 \cup c_2 \in H^3(C_n,\ZZ/n\ZZ)$ is a generator of this group. Thus, given a map $f \colon \pi_1(X,x) \rightarrow C_n$, we see that to determine the Kim invariant for all classes in $H^3(C_n,\ZZ/n\ZZ)$, it is enough to do it for $c$. We will compute Kim's invariant by computing the cup product $f_X^*(c) = f_X^*(c_1) \cup f_X^*(c_2) \in H^3(X,\ZZ/n\ZZ)$ using Lemma \ref{cor:cup1nonzero}. Note that $f_X^*(c_1) \in H^1(X,\ZZ/n\ZZ) \cong (\Cl K/n \Cl K)^\sim$ corresponds to a $\ZZ/n\ZZ$-torsor $p\colon Y \rightarrow X$ of the form $\Ind^{C_n}_{C_d}(\Spec \OO_L)$ for $L$ an unramified Kummer extension $L/K$ of degree $d|n$, say $L= K(v^{1/n})$. We now want to see what element $f_X^*(c_2) \in H^2(X,\ZZ/n\ZZ)$ corresponds to if we identify $H^2(X,\ZZ/n\ZZ)$ with the Pontryagin dual of $\Ext^1_X(\ZZ/n\ZZ,\mathbb{G}_{m,X})$, using Artin--Verdier duality.

\begin{lem} \label{lemma:bock}
Under the isomorphism $H^2(X,\ZZ/n\ZZ) \cong \Ext^1_X(\ZZ/n\ZZ,\mathbb{G}_{m,X})^\sim$,  the element $f_X^*(c_2) \in H^2(X,\ZZ/n\ZZ)$ corresponds to the element $w \in \Ext^1_X(\ZZ/n\ZZ,\mathbb{G}_{m,X})^\sim \cong H^1(\Hom(\mathcal{E}_n,\mathcal{C}))^\sim$, defined by $w(a,\mathfrak{a}) = \Art_{L|K}(\mathfrak{a})$, where $\Art_{L|K}\colon \Cl K/n \Cl K \rightarrow \ZZ/n\ZZ$ is the Artin symbol associated to the unramified extension $L/K$.
\end{lem}

\begin{proof}
Let us note that, by Artin reciprocity, the class $f_X^*(c_1) \in H^1(X,\ZZ/n\ZZ)$ corresponds to the Artin symbol $\Art_{L|K}.$ By functoriality, we need to show that the connecting homomorphism $\beta_n \colon H^1(X,\ZZ/n\ZZ) \rightarrow H^2(X,\ZZ/n\ZZ)$ of the short exact sequence 
  \begin{equation}\label{bockeq} 
    0 \rightarrow \ZZ/n\ZZ \xrightarrow{n} \ZZ/n^2 \ZZ_{X} \rightarrow \ZZ/n\ZZ \rightarrow 0\,,
  \end{equation} 
takes the class of $x$ to the functional $y \in H^2(X,\ZZ/n\ZZ) \cong H^1(\Hom(\mathcal{E}_n,\mathcal{C}))^\sim$ defined by $y(a,\mathfrak{a}) = \Art_{L|K}(\mathfrak{a}).$ Equivalently (by Artin--Verdier duality) we need to show that the connecting homomorphism $$\beta_n^\sim \colon \Ext^{1}_X(\ZZ/n\ZZ,\mathbb{G}_{m,X}) \rightarrow \Ext^2_X(\ZZ/n\ZZ,\mathbb{G}_{m,X}),$$ of the short exact sequence (\ref{bockeq}), takes $(a,\mathfrak{a}) \in \Ext^1_X(\ZZ/n\ZZ,\mathbb{G}_{m,X})$ to $\mathfrak{a} \in \Cl K/n \Cl K.$ To see this, consider first the diagram of exact sequences
  \[
    \begin{tikzcd} 
      0 \arrow[r] & \ZZ_{X} \arrow[d] \arrow[r,"n"]& \ZZ_{X} \arrow[r] \arrow[d] & \ZZ/n\ZZ \arrow[d,equals] \arrow[r] & 0 \\ 0 
      \arrow[r] & \ZZ/n\ZZ \arrow[r,"n"] & \ZZ/n^2\ZZ_{X} \arrow[r] & \ZZ/n\ZZ \arrow[r] & 0 \,.
    \end{tikzcd}
  \]
If we apply $\Ext_X(-,\mathbb{G}_{m,X})$, to the upper short exact sequence, we get a connecting homomorphism  $\beta^\sim \colon H^{i}(X,\mathbb{G}_{m,X}) \rightarrow \Ext^{i+1}_X(\ZZ/n\ZZ,\mathbb{G}_{m,K})$. It is then clear that if 
  \[
    \pi\colon\Ext^i_X(\ZZ/n\ZZ,\mathbb{G}_{m,X}) \rightarrow H^i(X,\mathbb{G}_m)
  \] 
is the map induced by the projection $\ZZ_X \rightarrow \ZZ/n\ZZ\,$, then $\beta_n^\sim = \beta^\sim \circ \pi.$ We now begin to compute $\beta_n^\sim$ by first noting that the map $\beta^\sim \colon H^1(X,\mathbb{G}_{m,X}) \rightarrow \Ext^2_X(\ZZ/n\ZZ,\mathbb{G}_{m,X})$ is the map which, under the isomorphisms $H^1(X,\mathbb{G}_{m,X}) \cong \Cl K$ and 
  \[
    \Ext^2_X(\ZZ/n\ZZ,\mathbb{G}_{m,X})\cong \Cl K/n\Cl K\,,
  \] 
takes $I \in \Cl K$ to its reduction in $\Cl K/n\Cl K.$ It is easy to see that 
  \[
    \pi \colon \Ext^1_X(\ZZ/n\ZZ,\mathbb{G}_{m,X}) \rightarrow H^1(X,\mathbb{G}_{m,X})
  \] 
is the map which, under the isomorphisms $\Ext^1_X(\ZZ/n\ZZ,\mathbb{G}_{m,X}) \cong H^1(\Hom(\mathcal{E}_n,\mathcal{C}))$ and $H^1(X,\mathbb{G}_{m,X}) \cong \Cl K$, takes the pair $(a, \mathfrak{b})$ in $\Ext^1_X(\ZZ/n\ZZ,\mathbb{G}_{m,X})$ to $\mathfrak{b} \in \Cl K$. By composition, we see that $\beta_n^\sim = \beta^\sim \circ \pi$ takes $(a,\mathfrak{b})$ to $\mathfrak{b} \in \Cl K/n \Cl K$ and we have our lemma.
\end{proof}

We can now apply Corollary \ref{cor:cup1nonzero} together with Lemma \ref{lemma:bock} to see that $f_X^*(c) = f_X(c_1) \cup f_X(c_2)$ is non-zero if and only if $\Art_{L|K}(\mathfrak{a}) \neq 0$, where $\mathfrak{a}$ is any fractional ideal such that $-\divis(v) = n\mathfrak{a}$. This allows us to give a criterion for when Kim's invariant vanishes.

\begin{prop}
Let $K$ be a totally imaginary number field containing a primitive $n$th root of unity and let $X= \Spec \OO_K$ be its ring of integers. Let $C_n=\ZZ/n\ZZ$, let $c_1 \in H^1(C_n,\ZZ/n\ZZ) \cong \Hom_{\Ab}(\ZZ/n\ZZ,\ZZ/n\ZZ)$ correspond to the identity and $c_2 = \beta_n(c_1) \in H^2(C_n,\ZZ/n\ZZ)$ be its image under the Bockstein homomorphism. Suppose that we have a continuous homomorphism $f\colon \pi_1(X,x) \rightarrow C_n$, corresponding to the unramified Kummer extension $L/K$, where $L=K(v^{1/n})$ for some $v \in K^\times$ such that there is an $\mathfrak{a} \in \Div K$ satisfying $n \mathfrak{a} = - \divis(v)$. Then Kim's invariant, $\inv(f_X^*(c)) \in \ZZ/n\ZZ$ vanishes if and only if $\mathfrak{a}$ is in the image of the norm map $N_{L|K} \colon \Cl L \rightarrow \Cl K$, i.e. if and only if $\Art_{L|K}(\mathfrak{a})=0$, where $\Art_{L|K}$ means the Artin symbol. 
\end{prop}

 \begin{proof}
 This follows immediately from the fact that $f_X^*(c) = f_X(c_1) \cup f_X(c_2)$ vanishes if and only if $\Art_{L|K}(\mathfrak{a})=0$, which holds if and only if $\mathfrak{a}$ is in the image of the norm map.  Since $\inv$ is an isomorphism, we then have that $\inv(f_X^*(c)) \neq 0$, so we have our proposition.
 \end{proof}

 \begin{rk}
 Formulas for this invariant have already been obtained in \cite{BleherUnramified} and \cite{KimAbelian}. Our formula differs from the one in \cite{BleherUnramified}, but the formula we gave can be derived from Proposition 3.10 in \cite{KimAbelian}, together with the quadratic form formula for the Chern--Simons invariant on page 7. The fact that our formula can be obtained from results in \cite{KimAbelian} was pointed out to us by Minhyong Kim.
\end{rk}

\bibliographystyle{dary}
\bibliography{cupproducts}{}
\end{document}